\documentclass[11pt,reqno]{amsart}
\usepackage{enumerate}
\usepackage{amsfonts, amsmath, amssymb, amscd, amsthm, bm}

\usepackage{enumerate}
\usepackage{url}
\usepackage[linktocpage=true,colorlinks,citecolor=magenta,linkcolor=blue,urlcolor=magenta]{hyperref}
\usepackage{multicol}
\usepackage{comment}

 \newtheorem{thm}{Theorem}[section]

 \newtheorem{lem}[thm]{Lemma}
 
 \newtheorem{prop}[thm]{Proposition}
 \theoremstyle{definition}
 \newtheorem{defn}[thm]{Definition}
 \newtheorem{rem}[thm]{Remark}

 \newtheorem{cla}[thm]{Claim}
 \numberwithin{equation}{section}
\numberwithin{equation}{section}

\newcounter{rom}
\renewcommand{\therom}{(\roman{rom})}
{\end{list}}

\title[Contracting hypersurfaces by  powers of the $\sigma_k$-curvature]%
{Contracting axially symmetric hypersurfaces  by powers of the $\sigma_k$-curvature}

\begin{document}

\author[H. Li]{Haizhong Li}
\address{Department of Mathematical Sciences,
Tsinghua University,  Beijing 100084, P. R. China}
\email{\href{mailto:hli@math.tsinghua.edu.cn}{hli@math.tsinghua.edu.cn}}

\author[X. Wang]{Xianfeng Wang}
\address{School of Mathematical Sciences and LPMC,
Nankai University,
Tianjin, 300071,  P. R. China; Mathematical Sciences Institute,
Australian National University, Canberra,
ACT 2601 Australia}
\email{\href{mailto:wangxianfeng@nankai.edu.cn}{wangxianfeng@nankai.edu.cn, xianfeng.wang@anu.edu.au}}


\author[J. Wu ]{Jing Wu}
\address{School of Mathematical Sciences and LPMC,
Nankai University,
Tianjin 300071,  P. R. China}
\email{\href{mailto:w3024v@gmail.com}{w3024v@gmail.com}}

\begin{abstract}
In this paper, we investigate the contracting curvature flow of closed, strictly convex axially symmetric hypersurfaces  in $\mathbb{R}^{n+1}$ and $\mathbb{S}^{n+1}$  by $\sigma_k^\alpha$, where $\sigma_k$ is the $k$-th elementary symmetric function of the principal  curvatures and $\alpha\ge 1/k$. We prove that for any $n\geq3$ and any fixed $k$ with $1\leq k\leq n$, there exists a constant  $c(n,k)>1/k$ such that that if $\alpha$ lies in the interval $[1/k,c(n,k)]$, then we  have a nice curvature pinching estimate involving the  ratio of the biggest principal curvature to the smallest principal curvature of the flow hypersurface, and we prove that the properly rescaled  hypersurfaces converge exponentially to the unit sphere.
In the case $1<k\le n \le k^2$, we can choose $c(n,k)=\frac{1}{k-1}$. Our results  provide an evidence for the general convergence result without initial curvature pinching conditions.
\end{abstract}


\maketitle

\section{Introduction}\label{sec1}
Let $\mathbb{R}^{n+1}(\epsilon)~(\epsilon=0,1,-1)$ be a real space form, i.e., when $\epsilon=0$, $\mathbb{R}^{n+1}(0)=\mathbb{R}^{n+1}$, when $\epsilon=1$, $\mathbb{R}^{n+1}(1)=\mathbb{S}^{n+1}$, and when $\epsilon=-1$, $\mathbb{R}^{n+1}(-1)=\mathbb{H}^{n+1}$.
Let $M$ be a smooth, closed manifold and $X_0:M\to\mathbb{R}^{n+1}(\epsilon)$ be a smooth immersion which is strictly convex,
we consider a smooth family of immersions $X(\cdot ,t):M \times [0,T) \to \mathbb{R}^{n+1}(\epsilon)$ solving the evolution equation
\begin{equation}\begin{aligned}\label{flow}
\left\{
\begin{array}{ll}
\frac{\partial X}{\partial t}(\cdot, t) = -\sigma_k^\alpha(\cdot, t)\nu(\cdot, t),\\
X(\cdot , 0) = X_0 (\cdot),
\end{array}
\right.
\end{aligned}\end{equation}
where $\alpha\geq 1/k$,  $\nu$ is the outer unit normal vector of $M_t=X_t(M)$ and $\sigma_k$ is the $k$-th elementary symmetric function of the principal  curvatures of $M_t$.
In particular, $\sigma_1$ is the mean curvature and $\sigma_n$ is the Gauss curvature.
Throughout this paper, we call \eqref{flow} a $\sigma_k^\alpha$-curvature flow, and we will consider two cases: $\epsilon=0$ and $\epsilon=1$. When $\epsilon=0$ and $k=1$, the flow \eqref{flow} was called $H^\alpha$-flow and studied by Schulze in \cite{SCHULZE2005,SCHULZE2006}. When $\epsilon=0$ and $k=2$, the flow \eqref{flow} corresponds to the flow by powers by the scalar curvature, which was studied by Alessandroni and  Sinestrari in \cite{AS2010}. When $\epsilon=0$ and $k=n$, the flow \eqref{flow} is the flow by powers of the Gauss curvature, which has been well studied, we refer to \cite{Andrews1996,A1999,A2000PJM,AC2012PAMQ,AGN2016,AL2015,BCD,GN2017} and the references therein.
\subsection{Some background of contracting curvature flows in Euclidean space and in sphere}\label{sec1.1}
When the ambient space is Euclidean space, there have been lots of results about contracting curvature flows. For the case $\alpha=k=1$, the flow is the well-known mean curvature flow which is the gradient flow of the area functional. In one of his famous papers, Huisken \cite{HUISKEN1984}  proved that for any convex initial hypersurface $M_0$, there exists a unique smooth solution to the mean curvature flow and the solution contracts to a ``round'' point in finite time. Similar results have been studied by  Chow for the flows by the n-th root of the Gauss curvature \cite{CHOW1985} and the square root of the scalar curvature \cite{CHOW1987} (with an initial pinching condition). Later on, by proving a geometrical pinching estimate,  Andrews \cite{A1994CVPDE} extended the results of Huisken and Chow to a wide class of curvature flows, with speeds given by homogeneous of degree $1$ functions of the principal curvatures and satisfying some natural conditions. In \cite{A2007} and \cite{A2010CVPED},  Andrews proved new powerful pinching estimates and improved the previous results to a much  wider class of curvature flows. In particular,  the results in \cite{A2007}  applied to the flow by square root of the scalar curvature, and Andrews removed the initial pinching condition in \cite{CHOW1987}.
In the previously mentioned papers, the speed functions of the flows are given by homogeneous of degree $1$ functions of the principal curvatures. For the flow by a speed function which is homogeneous of degree $\alpha>1$, there are fewer results. The first celebrated result was proved by  Andrews in \cite{A1999} for Gauss curvature flow, where  Firey's  conjecture that convex surfaces moving by their Gauss curvature become spherical as they contract to points was proved. Guan and Ni \cite{GN2017} proved that convex hypersurfaces in $\mathbb{R}^{n+1}$ contracting by the Gauss curvature flow  converge (after rescaling to fixed volume) to a smooth uniformly convex self-similar solution of  the flow. Andrews, Guan and Ni \cite{AGN2016} extended the results in \cite{GN2017}  to the flow by powers of the Gauss curvature $K^\alpha$ with $\alpha>\frac{1}{n+2}$.
  Brendle, Choi and Daskalopoulos \cite{BCD} proved that round spheres are the only closed, strictly convex self-similar solutions to the $K^{\alpha}$ flow with $\alpha> \frac{1}{n+2}$. Therefore, the generalized Firey's conjecture proposed by Andrews in \cite{Andrews1996} was completely solved, that is, the solutions of the flow by powers of the Gauss curvature  converge to spheres for any $\alpha>\frac{1}{n+2}$.

When the ambient space is the sphere, there are also some interesting results about contracting curvature flows. For the mean curvature flow in the sphere,  Huisken \cite{H1987} proved that if the initial hypersurface (not necessarily convex) satisfies a curvature pinching condition, then either the evolving hypersurfaces  converge uniformly to a single point in finite time, or the flow exists for all time and the evolving
hypersurfaces converge in $C^{\infty}$-topology to a smooth totally geodesic hypersurface.
 Andrews \cite{Andrews2002ICM} proved some optimal results for
 contracting curvature flows of surfaces with positive intrinsic curvature in  $\mathbb{S}^3$ in the sense that the weakest condition is required on the initial
 surfaces, by proving the existence of an optimal fully nonlinear speed function.  Gerhardt \cite{Gerhardt2015} established a dual relation between the contracting curvature flow and the expanding curvature flow for strictly convex hypersurfaces in the sphere by using the Gauss map, and  proved that if the speed function $F$ is homogeneous of degree $1$, concave and inverse-concave, then the flow hypersurfaces will shrink to a point in finite time, if $F$ is strictly concave, or $F=H/n$, then the properly rescaled hypersurfaces converge to the unit sphere exponentially.  Wei \cite{Wei17} proved similar conclusion for the case that $F$ is  homogeneous of degree $1$, concave and $F$ approaches zero on the boundary of the positive quadrant.  McCoy \cite{McCoy2017} proved that in the surface case, if the speed function is a homogeneous of degree $1$ function or the Gauss curvature, then strictly convex surfaces in  $\mathbb{S}^3$ will contract to round points in finite time, and the results were extended to strictly convex axially symmetric case for $n\geq 3$. For the surface case, very recently, Hu, Li, Wei and Zhou \cite{HLWZ2019} proved that the flow by a power of the mean curvature with the power $\alpha\in[1,5]$ and the flow by a power of the Gauss curvature with the power $\alpha\in[1/2,1]$ will both contract strictly convex surfaces in $\mathbb{S}^3$ to  round points in finite time.
\subsection{Two natural questions and the main theorems}\label{sec1.2}
Basing on the generalized Firey's conjecture mentioned above, it is natural to ask the following questions:
\textit{Question 1}. For any fixed $k$ with $1\leq k\leq n-1$, can the solutions of the $\sigma_k^\alpha$-curvature flow \eqref{flow} with closed, strictly convex initial hypersurfaces in $\mathbb{R}^{n+1}$ converge to round spheres after proper rescaling for some $\alpha>\frac{1}{k}$?
\textit{Question 2}. For any fixed $k$ with $1\leq k\leq n$, can the solutions of the $\sigma_k^\alpha$-curvature flow \eqref{flow} with closed, strictly convex initial hypersurfaces in $\mathbb{S}^{n+1}$ converge to round spheres after proper rescaling for some $\alpha>\frac{1}{k}$?

As far as the authors know, the above questions are open.
For \textit{Question 1}, the recent result of Gao, Li and Ma \cite{GLM} that  closed, strictly convex self-similar solutions to the $\sigma_k^{\alpha}$-curvature flow must be round spheres, provides a new understanding of the $\sigma_k^{\alpha}$-curvature flow.
In the case of the $\sigma_1^\alpha$-curvature flow, i.e., the $H^{\alpha}$-flow,  Schulze \cite{SCHULZE2005,SCHULZE2006} showed that for the $H^\alpha$-flow of a closed, strictly convex hypersurface in $\mathbb{R}^{n+1}$ with $\alpha\geq 1$,  if the initial ratio of the biggest and smallest principal curvatures at every point is close enough to $1$, depending only on $\alpha$ and $n$, then this is preserved under  the flow and the evolving hypersurfaces converge to the unit sphere in finite time after rescaling appropriately. In the appendix of \cite{SCHULZE2006}, Schulze and Schn\"urer showed that in the $2$-dimensional case, if $\alpha\in[1,5]$, no initial pinching condition is needed to guarantee that the properly  rescaled surfaces converge to the unit sphere. When $k=2$ and $\alpha>1/2$,  Alessandroni and Sinestrari \cite{AS2010} proved that if the initial hypersurface is strictly convex and satisfies a suitable pinching condition, then the solution shrinks to a point in finite time and converges to a sphere after a proper rescaling.
For flow of convex hypersurfaces by arbitrary speeds which are smooth homogeneous functions of the principal curvatures of degree greater than one, Andrews and McCoy \cite{AM2012TAMS} proved that for smooth strictly convex initial hypersurfaces with the ratio of principal curvatures sufficiently close to $1$ at each point, the flow hypersurfaces remain smooth and strictly convex and converge to round spheres in finite time after proper rescaling.
For \textit{Question 2}, the only related results are the results proved by McCoy \cite{McCoy2017}, Hu, Li, Wei and Zhou  \cite{HLWZ2019} mentioned above. If the initial hypersurface of the sphere is pinched enough, Li and Lv \cite{LILV2019} proved that the flow converges smoothly and exponentially to the unit sphere after suitable rescaling for some homogeneous functions of the principal curvatures of degree greater than one, which include the functions $\sigma_k^{\alpha}$ for $\alpha>1/k$. Li and Lv's result can be  regarded as a counterpart of the result by Andrews and McCoy \cite{AM2012TAMS}.

The aim of this paper is to find appropriate constants $c_0(n,k)>1/k$ and $c_1(n,k)>1/k$ which only depend on $n$ and $k$, such that:
(i)  For any fixed $k$ with $1\leq k\leq n-1$, if $\alpha\in[1/k,c_0(n,k)]$,  then any closed, strictly convex axially symmetric hypersurface   in $\mathbb{R}^{n+1}$  ($n\geq 3$) will contract to a round point under the  $\sigma_k^\alpha$-curvature flow without initial curvature pinching conditions.
(ii)  For any fixed $k$ with $1\leq k\leq n$, if $\alpha\in[1/k,c_1(n,k)]$,  then any closed, strictly convex axially symmetric hypersurface  in $\mathbb{S}^{n+1}$ ($n\geq 3$) will contract to a round point under the  $\sigma_k^\alpha$-curvature flow without initial curvature pinching conditions.
Our results provide an affirmative answer to the questions proposed above in axially symmetric case. More precisely, we prove the following results.

\begin{thm}\label{thm1.1}
Let $X_0:M\to \mathbb{R}^{n+1}$ be a smooth, closed, strictly convex axially symmetric hypersurface, $n\geq 3$, $1\leq k\leq n-1$. Then there exists a unique  smooth solution of
the $\sigma_k^{\alpha}$-curvature flow \eqref{flow} on a maximal finite time interval $[0,T)$ for $\alpha\geq 1/k$. For each $n$ and $k$, there exists a  constant  $c_0(n,k)>\frac{1}{k}$ such that if
$\alpha\in [1/k, c_0(n,k) ]$, then the flow hypersurfaces $M_t=X_t(M)$  are closed,  strictly convex, axially symmetric and converge to a point $q\in\mathbb{R}^{n+1}$ as $t\to T$,
and the rescaled embeddings
\begin{equation*}
\tilde{X}(p,t):= \big((k\alpha+1)\binom{n}{k}^\alpha (T-t)\big)^{-\frac{1}{k\alpha+1}}\big(X(p,t)- q\big)
\end{equation*}
converge exponentially in $C^{\infty}$ to the unit sphere $\mathbb{S}^n$ as $t \to T$.
\end{thm}

\begin{thm}\label{thm1.2}
Let $X_0:M\to \mathbb{S}^{n+1}$ be a smooth, closed, strictly convex axially symmetric hypersurface, $n\geq 3$, $1\leq k\leq n$.  Then there exists a unique  smooth solution of
the $\sigma_k^{\alpha}$-curvature flow \eqref{flow} on a maximal finite time interval $[0,T)$ for $\alpha\geq 1/k$. For each $n$ and $k$, there exists a  constant  $c_1(n,k)>\frac{1}{k}$ such that if
$\alpha\in [1/k,c_1(n,k)]$, then the flow hypersurfaces $M_t=X_t(M)$  are closed,  strictly convex, axially symmetric and converge to a point $q\in\mathbb{S}^{n+1}$ as $t\to T$,
and the properly rescaled hypersurfaces
converge exponentially in $C^{\infty}$ to the unit sphere $\mathbb{S}^n$ as $t$ approaches $T$ in the following sense: We denote by $\Theta(t,T)$ the sphere solution of the flow \eqref{flow} which shrinks to a point when $t\to T$. If we introduce geodesic polar coordinates with center $q$, write the flow hypersurface $M_t$ as a graph of a function $u(p,t)$ over $\mathbb{S}^n$, and define a new time parameter $\tau=-\log{\Theta(t,T)}$, then $\tau$ tends to $\infty$ as $t\to T$ and the rescaled function $\tilde{u}(p,\tau)=u(p,t)\Theta(t,T)^{-1}$ is uniformly bounded and converges exponentially in  $C^{\infty}$ to the constant function $1$ as $\tau\to \infty$.
\end{thm}
\begin{rem}\label{rem1.3}
(i) For  Theorem \ref{thm1.1}, when $k=1,~\alpha=1$, the result is due to Huisken \cite{HUISKEN1984}. When $1<k<n,~\alpha=\frac{1}{k}$, the result is due to Andrews \cite{A2007}.
	(ii) When $\alpha=\frac{1}{k}$, the result in Theorem \ref{thm1.2} is a special case of the results in \cite{Gerhardt2015} or \cite{Wei17}. For this reason, we will consider the case $\alpha\in(1/k,c_{\epsilon}(n,k)]$ ($\epsilon=0,1$) in the proof of the last part of Theorem \ref{thm1.1} and Theorem \ref{thm1.2}.
\end{rem}
\begin{rem}
		Although we can not write down the constants $ c_0(n,k) $ and $c_1(n,k)$ in terms of explicit functions of $n$ and $k$, they can be precisely determined by  applying Sturm's theorem. We list some of the values of $ c_\epsilon(n,k) $  for $\epsilon=0,1$. For example, $c_\epsilon(3,1)=3.64...,~ c_\epsilon(4,1)=2.93...,$ etc.
In the case $1<k\le n \le k^2$, we can choose $ c_\epsilon(n,k)=\frac{1}{k-1}$ for $\epsilon=0,1$.
For general $n$ and $\epsilon=0,1$, we prove that $ c_\epsilon(n,1)\ge 1+\frac{7}{n}$ and $ c_\epsilon(n,k) \geq \frac{1}{k}+\frac{k}{(k-1)n}$ for $k\ge 2$ and $n\geq k^2$. The proof is given in the appendix.
\end{rem}
\subsection{Outline of the proof and organization of the paper}\label{sec1.3}
In \S\ref{sec2}, we give some notations and preliminary results.
\S\ref{sec3} is devoted to  proving a curvature pinching estimate \eqref{3.1}, which is the key step in the proof of our main theorems.
The main idea to prove \eqref{3.1} is to apply the maximum principle to the evolution equation for the quantity defined by
 \begin{equation}\label{defG}
 G=\sigma_k^{2\alpha}\sum_{i<j}(\frac{1}{\lambda_i}-\frac{1}{\lambda_j})^2
 \end{equation}
 under the flow \eqref{flow}, where $\lambda_i$ are the principal curvatures of the flow hypersurfaces.
 This is inspired by \cite{AL2015}, where  Andrews and  the first author considered the evolution of $K^{2\alpha}\sum_{i<j}(\frac{1}{\lambda_i}-\frac{1}{\lambda_j})^2
 $ for the flow by  powers of Gauss curvature. Note that $G$ can be written in the following form: $G=\sigma_k^{2\alpha}\cdot\big((n-1)\sigma_{n-1}^2-2n\sigma_n\sigma_{n-2}\big)/\sigma_n^2$, which is clearly a smooth symmetric curvature function.
We will prove in Theorem \ref{thm3.1} that for $\epsilon=0,1$, $n\geq3$ and any fixed $k$ with $1\leq k\leq n$, there exists a constant $c_{\epsilon}(n,k) $ such that if $\alpha\in[1/k, c_{\epsilon}(n,k) ]$, then the maximum of the quantity $G$
 is non-increasing in time.
 The proof of \eqref{3.1} comprises three steps.
In the first step, we prove a positive lower bound for the
  $\sigma_k$-curvature of the flow hypersurfaces, by applying maximum principle to
  the evolution of $\sigma_k$ under the flow (see Lemma \ref{lem3.1}). Theorem \ref{thm3.1} is the second step. The uniform upper bound on $\sigma_k$ in combination with the uniform  upper bound on $G$ obtained in Theorem \ref{thm3.1} leads to uniform  lower and upper bounds \eqref{3.14} on the ratio of the maximal principal curvature to the minimal principal curvature on the flow hypersurface $M_t$. In the last step, armed with \eqref{3.14}, we obtain \eqref{3.1} by using Theorem \ref{thm3.1} again.
As a consequence of \eqref{3.14}, we obtain that if $\alpha\in[1/k, c_{\epsilon}(n,k) ]$, then
the strict convexity of the flow hypersurface is preserved under the flow \eqref{flow} for $\epsilon=0,1$.
 The proof of Theorem \ref{thm3.1} is given in \S\ref{sec3.2}, and we discuss Euclidean case and the sphere case separately. The gradient terms of the
 evolution of $G$ are same for both cases: $\epsilon=0$, $\epsilon=1$. By a long calculation, we obtain that at a spatial critical point of $G$, the sign
 of the gradient terms is the same as the sign of a sextic polynomial $Q$ defined by \eqref{Q}. By applying Strum's theorem, we can find the desired constant $c_0(n,k) $ such that if $\alpha\in[1/k, c_0(n,k) ]$, then $Q$ is non-positive for all positive variables $x$, which implies that the gradient terms  of the evolution of $G$ are non-positive at any spatial critical point. Since the procedure of applying Strum's theorem to to determine and estimate the constant $c_0(n,k)$ is long and technical, we give the details of this part in the appendix.
 The zero-order terms of the evolution of $G$ for Euclidean case are automatically zero, while the zero-order terms for the sphere case can be proved to be non-positive if $\alpha\in[1/k,c_2(n,k)]$, with $c_2(n,k)$ given by \eqref{c2nk}. We define $c_1(n,k)=\min\{c_0(n,k),~c_2(n,k)\}$. Thus, we have found the constant $c_{\epsilon}(n,k) $ which satisfies that if $\alpha\in[1/k, c_{\epsilon}(n,k) ]$, then both the zero-order terms and gradient terms of the evolution of $G$ at a spatial critical point are  non-positive, so we can apply the parabolic maximum principle to complete the proof of Theorem \ref{thm3.1}.

In \S\ref{sec4}, we complete the proof of Theorems \ref{thm1.1}-\ref{thm1.2}. We already obtained that the maximal existence time $T$ of the flow \eqref{flow} is finite in Lemma \ref{lem3.1}.
By an analogous
argument to that in \cite[\S 3]{SCHULZE2006} (for the case $\epsilon=0$) and \cite[\S 6]{Gerhardt2015} (for the case $\epsilon=1$), the pinching estimate \eqref{3.14} implies an upper bound for the ratio of the outer radius $\rho_+(t)$ to the inner radius $\rho_-(t)$ of the flow hypersurface
$M_t$ for $t\in[0,T)$ in the case $\epsilon=0$ and for $t\in[t_\delta,T)$ for the case $\epsilon=1$, where $t_\delta=T-\delta$ and $\delta$ is sufficiently small.
Then we can use a technique  of Tso \cite{Tso1985} to prove that the $\sigma_k$-curvature remains bounded from above  as long as
the flow \eqref{flow} bounds a non-vanishing volume, which together with the pinching estimate \eqref{3.14} implies a uniform upper bound for the principal curvatures.
Since the flow hypersurface $M_t$ is also uniformly strictly convex, we obtain that the flow \eqref{flow} remains to be uniformly parabolic.
Since the speed function can be written in the form $F=\sigma_k^{\alpha}=(\sigma_k^{1/k})^{k\alpha}$ and $\sigma_k^{1/k}$ is a concave function of the principal curvatures,
we can apply the  H\"{o}lder estimate by Andrews \cite[Theorem 6]{andrews2004fully} (we can also apply the H\"{o}lder estimate in the case of one space dimension in  \cite{lieberman1996second}, since
axially symmetric hypersurface can be written
as a graph on the unit sphere in geodesic polar coordinates and the graph function  has only one space variable) and
parabolic Schauder estimate \cite{lieberman1996second} to get uniform $C^\infty$ estimates of the solution, hence the solution can be extended beyond $T$, which contradicts the maximality of $T$.
Therefore, we obtain that both the inner radius and outer radius converge to $0$ as $t\to T$, so the flow hypersurfaces remain smooth until they shrink to a point.

We deal with the rescaling in Euclidean case and in sphere case in \S\ref{sec4.2} and \S\ref{sec4.3} respectively.
As remarked in Remark \ref{rem1.3}, we consider the case $\alpha\in(1/k,c_{\epsilon}(n,k)]$, where $c_{\epsilon}(n,k)$ is the constant in Theorem \ref{thm3.1}.
In Euclidean case, we rescale the flow hypersurfaces by $\tilde{X}(p,t):= (X(p,t)-q)\rho(t,T)^{-1}$, where $q$ is the point $M_t$ shrinks to,  $T$ is the maximal existence time of the flow \eqref{flow} and $\rho(t,T)$ is the radius of the sphere solution of the flow \eqref{flow} with center $q$ and maximal existence time $T$. We define a new time parameter $\tau$ by \eqref{tau}. We first apply the technique of Tso \cite{Tso1985} to
obtain a uniform upper bound for the $\tilde{\sigma}_k$-curvature of the rescaled hypersurface $\tilde{M}_t$. When $\alpha>1/k$, the coefficient of the second order part $(\dot{\tilde{\sigma}}_k^\alpha)^{ij}\nabla_i\nabla_j\tilde{\sigma}_k$ in the evolution equation of $\tilde{\sigma}_k$ will becomes degenerate if $\tilde{\sigma}_k$ is sufficiently small. Since we don't know of a suitable parabolic  Harnack inequality
for the flow \eqref{flow} to help us to obtain a positive lower bound for $\tilde{\sigma}_k$, we can not apply the H\"{o}lder estimate by Andrews \cite{andrews2004fully} or
the H\"{o}lder estimate in the case of one space dimension in  \cite{lieberman1996second}
 immediately  to get $C^{2,\alpha}$ estimates.
We will apply the interior H\"{o}lder estimates due to DiBenedetto and Friedman \cite{DF1985} to get  H\"{o}lder continuity of $\tilde{\sigma}_k$, by writing the evolution equation of $\tilde{\sigma}_k$-curvature of $\tilde{M}_t$ in a special form.  Finally, we  obtain that the rescaled flow  hypersurfaces converge  in $C^{\infty}$-topology to the unit sphere $\mathbb{S}^n$, by using analogous argument to that in  \cite{SCHULZE2006} and replacing the estimate (2.3) in Theorem 2.6 of \cite{SCHULZE2006} by  our  pinching estimate \eqref{3.1}.
By considering the evolution of the rescaled quantity $\tilde{G}$, we obtain that the maximal principal curvatures approache the minimal principal curvatures  exponentially
fast on the rescaled hypersurfaces. Then the exponential convergence of the rescaled hypersurfaces can be proved by standard arguments as done in \cite{A1994CVPDE} and \cite{SCHULZE2006}.

In the sphere case, we use a similar rescaling to that in \cite{Gerhardt2015}. We denote by $\Theta(t,T)$ the radii of the sphere solution which shrinks to a point as $t\to T$, where $T$ is the maximal existence time of the flow \eqref{flow} with initial hypersurface $M_0$ for $\epsilon=1$.
Let $q\in\mathbb{S}^{n+1}$ be the point that the flow hypersurfaces $M_t$ shrink to as $t$ approaches $T$, we introduce geodesic polar coordinates with center $q$.
We define a new time parameter by $\tau=-\log{\Theta(t,T)}$.
We prove that the rescaled function $\tilde{u}(p,\tau)=u(p,t)\Theta(t,T)^{-1}$   converges exponentially in  $C^{\infty}$ to the constant function $1$ as $\tau\to \infty$.
There are two key steps in the proof. First, due to a similar reason to Euclidean case, we can not apply the Harnack inequality as in \cite{Gerhardt2015} to obtain positive lower bound for $\tilde{\sigma}_k$ and to ensure uniform parabolicity, we use similar method to that in Euclidean case to obtain a uniform upper bound and H\"{o}lder continuity for $\tilde{\sigma}_k=\sigma_k\cdot\Theta(t,T)^k$. Second, we use our key estimate \eqref{3.1}, the bound on the ratio of outer radius to the inner radius \eqref{pinchingrho} together with the uniform upper bound and H\"{o}lder continuity for $\tilde{\sigma}_k$ to prove that $\tilde{u}(\cdot,\tau)$ obeys uniform a priori estimates in $C^{\infty}(\mathbb{S}^n)$ independently of $\tau$. Finally, by a similar argument to that in Section 8 of \cite{Gerhardt2015}, we obtain that $\tilde{u}(p,\tau)$ converges exponentially fast to the constant function $1$ in $C^\infty$-topology as $\tau\to \infty$.

\textbf{Acknowledgments:} The authors would like to thank Professor Ben Andrews and Dr. Yong Wei for their interest and helpful discussions.
The first author was supported in part by NSFC Grant No.11671224, No.11831005 and NSFC-FWO 11961131001.
The second and third authors were supported in part by NSFC Grant No.11571185 and the Fundamental Research Funds for the Central Universities.
X. Wang would also like to express her deep gratitude to the Mathematical Sciences Institute at the Australian National University for its hospitality and to Professor  Ben Andrews for his encouragement and help during her stay in MSI of ANU as a Visiting Fellow, while part of this work was completed.

\section{Notations and preliminaries}\label{sec2}
In this section, we give some notations and preliminary results.
Throughout the paper, we use the Einstein summation convention of sum over repeated indices.
Let $X_t=X(\cdot,t):M_t\to\mathbb{R}^{n+1}(\epsilon)$ be a family of hypersurfaces moving according to the $\sigma_k^\alpha$-curvature flow \eqref{flow}. We  use  $g=\{g_{ij}\}$, $A=\{h_{ij}\}$ and $\mathcal{W}=\{h^i_j\}$ to denote  the components of  induced metric,  the  second fundamental form and the Weingarten map of the hypersurfaces, respectively. In local coordinates $\xi^1,\cdots,\xi^n$, we can write $h_{ij}=-\bar{g}(\bar{\nabla}_{\frac{\partial X}{\partial \xi^i}}\frac{\partial X}{\partial \xi^j},\nu)$, where $\bar{g}$ denotes the metric of $\mathbb{R}^{n+1}(\epsilon)$, $\bar{\nabla}$ denotes the Levi-Civita connection with respect to the metric $\bar{g}$ and $\nu$ is the outer unit normal.  We denote  the principal curvatures of the hypersurface by $\lambda_1,\cdots,\lambda_n$, then the $\sigma_k$-curvature is defined by
\begin{equation}
\sigma_k=\sum_{1\le i_1 \le \cdots \le i_k \le n} \lambda_{i_1} \cdots \lambda_{i_k}.
\end{equation}
When $k=1$, $\sigma_1$ is the mean curvature. When $k=n$, $\sigma_n$ is the Gauss curvature.

\subsection{Properties of symmetric curvature functions}\label{sec2.1}
Let $F=F(A,g)=F(\mathcal{W})=F(\lambda(\mathcal{W}))$ be a smooth, symmetric function of the principal curvatures of a hypersurface $M\in\mathbb{R}^{n+1}(\epsilon)$, $F$ can be considered as a function of $\mathcal{W}=(h_i^j)$ or the principal curvatures  $\lambda(\mathcal{W})=(\lambda_1,\cdots,\lambda_n)$.  We denote by $(\dot{F}^{ml})$ and  $(\ddot{F}^{ml,rs})$ the matrices of the first and second partial derivatives of $F$ with respect to the components of its first arguments:
\begin{equation*}
\begin{aligned}
\frac{\partial }{\partial s}F(A+sB,g)|_{s=0}&=\dot{F}^{ml}B_{ml},\\
\frac{\partial ^2}{\partial s^2}F(A+sB,g)|_{s=0}&=\ddot{F}^{ml,rs}B_{ml}B_{rs}.
\end{aligned}
\end{equation*}
If $A$ is a diagonal with distinct eigenvalues and $B$ is a symmetric matrix, then we have the following relation (cf. \cite{A1994CVPDE})
\begin{equation}\label{2.1}
\ddot{F}^{ml,rs}(A)B_{ml}B_{rs}=~\frac{\partial^2F(\lambda(A))}{\partial\lambda_m\partial\lambda_l}B_{mm}B_{ll}+2\sum_{m<l}\frac{\frac{\partial F}{\partial \lambda_m}-\frac{\partial F}{\partial \lambda_l}}{\lambda_m-\lambda_l}B_{ml}^2.
\end{equation}
The second term in \eqref{2.1} makes sense as a limit if $\lambda_m=\lambda_l$.

If $F$ is a homogeneous of degree $\beta$  function of the principal curvatures $\lambda_1,\cdots,\lambda_n$, we have the following relations by using Euler's Theorem:
\begin{equation}\label{2.2}
\begin{aligned}
\dot{F}^{ij}h_{ij}=\beta F,~\dot{F}^{ij,ml}h_{ij}h_{ml}=\beta(\beta-1)F.
\end{aligned}
\end{equation}
We collect some properties of $\sigma_k$-curvature for later use.

\begin{lem} \label{lem2.1}For $\sigma_k$-curvature function with $1\leq k\leq n$, we have the following properties.
\begin{itemize}
\item[(i)] $\frac{\partial\sigma_k}{\partial\lambda_i}>0$ for all $i=1,\cdots,n$ and $(\lambda_1,\cdots,\lambda_n)\in\Gamma_k$, where $\Gamma_k$ is the connected component of
$\{(\lambda_1,\cdots,\lambda_n)\in\mathbb{R}^n:\sigma_k(\lambda_1,\cdots,\lambda_n)>0\}$ containing the positive cone.
\item[(ii)] $\sigma_k^{1/k}$ is concave and inverse concave in $\Gamma_+$.  We say that a curvature function $f$ is inverse concave, if the dual function of $f$ defined by $f_{*}(x_1,\cdots,x_n)=f(\frac{1}{x_1},\cdots,\frac{1}{x_n})^{-1}$ is concave.
\item[(iii)] $\sigma_k^{1/k}(1,\cdots,1)=\binom{n}{k}^{1/k}$ and $(\dot{\sigma_k^{1/k}})^{ml}h_{mr}h^r_l\geq \frac{\sigma_k^{2/k}}{\binom{n}{k}^{1/k}}$ in $\Gamma_+$.
Consequently,  $$\dot{\sigma_k}^{ml}h_{mr}h_l^r\geq \frac{k}{\binom{n}{k}^{1/k}}(\sigma_k)^{1+\frac{1}{k}},~\dot{(\sigma_k^{\alpha})}^{ml}h_{mr}h_l^r\geq \frac{k\alpha}{\binom{n}{k}^{1/k}}(\sigma_k^{\alpha})^{1+\frac{1}{k\alpha}} ~\text{in} ~\Gamma_+.$$
\item[(iv)]  $\nabla_i\big(\dot{\sigma_k}^{ij}\big)=0$ for any $j=1,\cdots,n$, where $\nabla$ is the Levi-Civita connection of the hypersurface $M\subset\mathbb{R}^{n+1}(\epsilon)$.
\end{itemize}
\end{lem}
\begin{proof}
Property (i) can be found in  \cite[Lemma 15.14]{lieberman1996second}. For Property (ii), the concavity of $\sigma_k^{1/k}$ can be found in \cite[Theorem 15.16]{lieberman1996second} and the inverse-concavity of $\sigma_k^{1/k}$ can be found in \cite[\S 2]{A2007}. Property (iii) follows from the inverse-concavity of $\sigma_k^{1/k}$ and Lemma 5 of
\cite{AMZ2013}. Property (iv) is a well-known property for
hypersurfaces in space forms, we refer to Proposition 2.1 of \cite{Reilly1973} and Lemma 3.1 of \cite{CL2007} for the proof.
\end{proof}

\subsection{Graphical representation for star-shaped hypersurfaces in the sphere}\label{sec2.2}
We recall the warped product model of the unit sphere $\mathbb{S}^{n+1}=I\times\mathbb{S}^n$ equipped with the warped product metric
\begin{equation*}
  \bar{g}=dr^2+\sin^2{r}~g_{\mathbb{S}^n},
\end{equation*}
where $I=(0,\pi)$.
Suppose that $M$ is a star-shaped hypersurface in $\mathbb{S}^{n+1}$ and can be expressed as a graph over the sphere $\mathbb{S}^n$, i.e., $ M=\{(u(\theta),\theta),~ \theta\in\mathbb{S}^n\}$ for some function $u\in C^{\infty}(\mathbb{S}^n)$, then the induced metric on $M$ in terms of the coordinates $\theta^j$ is given by
\begin{equation*}
  g_{ij}=u_iu_j+\sin^2{u}~\sigma_{ij},
\end{equation*}
where $\sigma_{ij}=g_{\mathbb{S}^n}(\partial_{\theta^i},\partial_{\theta^j})$ are the components of the round metric $g_{\mathbb{S}^n}$.
The second fundamental form $h_{ij}$ satisfies
\begin{equation*}
  h_{ij}v^{-1}=-u_{ij}+\sin{u}\cos{u}\sigma_{ij},
\end{equation*}
where $u_{ij}$ are the covariant derivatives of $u$ with respect to the induced metric $g_{ij}$ and $v$ is defined by
\begin{equation}\label{s2:v-def}
  v=\sqrt{1+\sin^{-2}u|Du|^2_{g_{\mathbb{S}^n}}}.
\end{equation}
The unit normal vector field on $M$ is given by
\begin{equation*}
  \nu=v^{-1}(\partial_r-\sin^{-2}u\cdot u^j\partial_{\theta^j}),~\text{with}~u^j=\sigma^{ij}u_i.
\end{equation*}
We define
 \begin{equation*}
   \varphi(u)=\int_{u_0}^u \frac 1{\sin{r}}dr,
 \end{equation*}
then $\varphi'(u)=\sin^{-1}u$, $g^{ij}=\sin^{-2}{u}(\sigma^{ij}-\varphi^i\varphi^j/v^2)$, and $h_j^i$ (the components of the  Weingarten map) can be expressed by
\begin{equation}\label{hij}
  h_j^i=v^{-1}\sin^{-1}u\left(-(\sigma^{ik}-v^{-2}\varphi^i\varphi^k)\varphi_{jk}+\cos{u}\delta_j^i\right),
\end{equation}
where $\varphi^i=\sigma^{ik}\varphi_k$, $(\sigma^{ij})=(\sigma_{ij})^{-1}$ and the covariant derivatives are taken with respect to $\sigma_{ij}$.

If $M_t$ is a smooth star-shaped solution of \eqref{flow} for $t\in[0,T)$ and each flow hypersurface is expressed as a graph $M_t=\text{graph } u(\theta,t)$ over the sphere $\mathbb{S}^n$, we can deduce that the defining function $u(\theta,t)$ of $M_t$ satisfies the following scalar parabolic equation (see \cite{Gerhardt2006})
\begin{equation}\label{s2:ut-evl}
  \frac{\partial}{\partial t} u(t)=-\sigma_k^\alpha v,
\end{equation}
on $[0,T)\times \mathbb{S}^n$, where $v$ is the function defined by \eqref{s2:v-def}.
 Let $\chi=\bar{g}(\sin{u}~\partial_r,\nu)$  denote the support function of $M_t$, we have  the following evolution equation (cf. \cite{Gerhardt2015,LWW}).
\begin{equation}\label{s2:evl-chi}
  \frac{\partial}{\partial t}\chi-\mathcal{L} \chi=\dot{(\sigma_k^{\alpha})}^{ml}h_{mr}h_l^r\cdot\chi-(k\alpha+1)\sigma_k^{\alpha}\cos{u}.
\end{equation}

\subsection{Evolution equations of curvature functions}\label{sec2.3}
For  hypersurfaces of $\mathbb{R}^{n+1}(\epsilon)$ moving according to the $\sigma_k^\alpha$-curvature flow \eqref{flow}, we have the following evolution equations (cf.   \cite{A1994R}, \cite{SCHULZE2006},\cite{McCoy2017}):
\begin{equation}\label{2.4}
\frac{\partial}{\partial t}g_{ij}=-2\sigma_k^{\alpha}h_{ij},
\end{equation}
\begin{equation}\label{nu}
\frac{\partial}{\partial t}\nu=\nabla(\sigma_k^{\alpha}),
\end{equation}
\begin{equation}\label{vol}
\frac{\partial}{\partial t}d\mu=-\sigma_k^{\alpha} H d\mu, ~\text{where}~d\mu=\sqrt{\det{g_{ij}}}d\xi^1\wedge\cdots \wedge d\xi^n,
\end{equation}
\begin{equation}\label{2.5}
\frac{\partial}{\partial t}\sigma_k^{\alpha} = \mathcal{L} \sigma_k^{\alpha}+\sigma_k^{\alpha}\dot{(\sigma_k^{\alpha})}^{ml}h_{mr}h_l^r+\epsilon\sigma_k^{\alpha}\dot{(\sigma_k^{\alpha})}^{ij}g_{ij},
\end{equation}
\begin{equation}\label{2.6}
\begin{aligned}
\frac{\partial}{\partial t}h^i_j=&\mathcal{L}h^i_j+\ddot{(\sigma_k^{\alpha})}^{ml,rs}\nabla^i h_{ml}\nabla_j h_{rs}-( k \alpha -1)\sigma_k^{\alpha}h^{im}h_{mj}+\dot{(\sigma_k^{\alpha})}^{ml}h_{mr}h_l^rh_j^i\\
&+\epsilon\big((1+k\alpha)\sigma_k^{\alpha}\delta^i_j-\dot{(\sigma_k^{\alpha})}^{ml}g_{ml}h^i_j\big),
\end{aligned}
\end{equation}
where $\nabla$ is the Levi-Civita connection with respect to the induced metric $g$, $\mathcal{L}= \dot{(\sigma_k^{\alpha})}^{ml}\nabla_m\nabla_l$, and $\nabla_ih_{ml}$ are the covariant derivatives of the second fundamental form.

It follows from Lemma 4.3 of \cite{AM2012TAMS} and Lemma 2.2 of \cite{McCoy2017} (cf. \cite{A1994R}) that  for any smooth symmetric function $G=G(\mathcal{W})=G(\lambda(\mathcal{W}))$, we have the following evolution equation for $G$ under the flow \eqref{flow}.
\begin{equation}\label{2.7}
\begin{aligned}
\frac{\partial}{\partial t} G=&\mathcal{L} G+(\dot{G}^{ij}\ddot{F}^{ml,rs}-\dot{F}^{ij}\ddot{G}^{ml,rs})\nabla_ih_{ml}\nabla_jh_{rs}\\
&+\dot{F}^{ml}h_{mr}h^r_l\dot{G}^{ij}h_{ij}+(1-k\alpha)F\dot{G}^{ij}h_{im}h^m_j\\
&+\epsilon\big((1+k\alpha)F \dot{G}^{ij}g_{ij}-\dot{F}^{ml}g_{ml}\dot{G}^{ij}h_{ij}\big),
\end{aligned}
\end{equation}
where $F$ is the speed function $\sigma_k^{\alpha}$ of the flow \eqref{flow}.

\subsection{Properties of axially symmetric hypersurfaces}\label{sec2.4}
Since the flow \eqref{flow} preserves symmetry, if $\{M_t\}$ is a solution of \eqref{flow} with  an axially symmetric initial hypersurface, then $M_t$ is also an axially symmetric hypersurface for each $t$.
An axially symmetric hypersurface (which is also called rotation hypersurface in the literature, cf. \cite{CD1983}) has at most two different principal curvatures, without loss of generality, we assume that $\lambda_1$ is the ``axial curvature" and $\lambda_2=\cdots=\lambda_n$ are the ``rotational curvatures", and denote the corresponding eigenvectors by $e_1,\cdots,e_n$.  When $n\geq 3$, the only possible nonzero components of the covariant derivatives of the second fundamental form are the following terms (cf. \cite{AHL}, \cite{MMW2015}).
\begin{equation}\label{2.8}
\nabla_1h_{11},~\nabla_1h_{22}=\cdots=\nabla_1h_{nn}.
\end{equation}

\section{The pinching estimates}\label{sec3}
\subsection{A key estimate}\label{sec3.1}
In this section, we will prove that: For any $n\geq3$ and any fixed $k$ with $1\leq k\leq n$, for $\epsilon=0,1$, there exists a constant $c_\epsilon(n,k) $ such that if $\alpha\in[1/k, c_\epsilon(n,k) ]$, then there exists a positive constant $C$ which only depends on the initial hypersurface $M_0$ in $\mathbb{R}^{n+1}(\epsilon)$ such that
\begin{equation}\label{3.1}
0 \le \frac{\lambda_{max}(p,t)}{\lambda_{min}(p,t)}+\frac{\lambda_{min}(p,t)}{\lambda_{max}(p,t)}-2 \le \frac{C}{\sigma_k^{2(\alpha-\frac{1}{k})}(p,t)},
\end{equation}
 for all  $(p,t) \in M \times [0,T)$, where $T$ is the maximal time of the solution of the flow \eqref{flow}. This is the key step in  the proof of Theorems \ref{thm1.1}-\ref{thm1.2}.

First, we prove a uniform positive lower bound for the $\sigma_k$-curvature of the flow hypersurfaces $M_t$.
For strictly convex initial hypersurface, the flow \eqref{flow} is uniformly parabolic  and has a unique smooth solution for at least a short time interval, by short time existence theorem (cf. \cite{HA1999}). By using the evolution equation \eqref{2.5}, we have the following evolution equation for $\sigma_k$,
\begin{equation}\label{sigmakevo}
\begin{aligned}
\frac{\partial}{\partial t}\sigma_k &=\dot{\sigma_k}^{ml}\nabla_m\nabla_l\sigma_k^{\alpha}
+\sigma_k^{\alpha}\dot{\sigma_k}^{ml}h_{mr}h_l^r+\epsilon\sigma_k^{\alpha}\dot{\sigma_k}^{ij}g_{ij}\\
&= \dot{(\sigma_k^{\alpha})}^{ml}\nabla_m\nabla_l\sigma_k+\alpha(\alpha-1)\sigma_k^{\alpha-2}\dot{\sigma_k}^{ml}\nabla_m \sigma_k\nabla_l \sigma_k +\sigma_k^{\alpha}\dot{\sigma_k}^{ml}h_{mr}h_l^r+\epsilon\sigma_k^{\alpha}\dot{\sigma_k}^{ij}g_{ij}.
\end{aligned}
\end{equation}
When $\epsilon=0,1$, by applying the maximum principle, it follows from \eqref{sigmakevo} that the minimum of $\sigma_k$ is increasing under the flow, that is, $\min_{M_t}\sigma_k\geq \min_{M_0}{\sigma_k}>0$.
When $\epsilon=0,1$, $\alpha\geq 1/k$, using the inequality in Lemma \ref{lem2.1} (iii), we obtain
 from \eqref{sigmakevo} that
 \begin{equation}\label{sigmakevo1}
\begin{aligned}
\frac{d}{d t}\min_{M_t}\sigma_k \geq\frac{k}{\binom{n}{k}^{1/k}}(\min_{M_t}\sigma_k)^{\alpha+1+1/k}.
\end{aligned}
\end{equation}
By applying maximum principle to \eqref{sigmakevo1}, for $\epsilon=0,1$, we have the following estimates.

\begin{lem} \label{lem3.1}
  \begin{equation*}
  \sigma_k(p,t)\geq \min_{M_0}{\sigma_k}(\cdot,0)\Big(1-\frac{k\alpha+1}{\binom{n}{k}^{1/k}}(\min_{M_0}{\sigma_k}(\cdot,0))^{\frac{k\alpha+1}{k}} \cdot t\Big)^{-\frac{k}{k\alpha+1}}.
  \end{equation*}
 Consequently, we obtain a finite upper bound for the maximal existence time:
  \begin{equation*}
  T\leq \frac{\binom{n}{k}^{1/k}}{k\alpha+1}(\min_{M_0}{\sigma_k}(\cdot,0))^{-\frac{k\alpha+1}{k}}.
  \end{equation*}
\end{lem}

In order to prove \eqref{3.1},  we need the following Theorem.
 \begin{thm}\label{thm3.1}
 Let $M_t$ be a family of smooth, closed, strictly convex, hypersurfaces in $\mathbb{R}^{n+1}(\epsilon)~(\epsilon=0,1)$, given by the $\sigma_k^\alpha$-curvature flow \eqref{flow}. We assume that $M_0$ is axially symmetric. For any $n\geq3$ and any fixed $k$ with $1\leq k\leq n$, there exists a constant $c_{\epsilon}(n,k) $ such that if $\alpha\in[1/k, c_{\epsilon}(n,k) ]$, then the maximum of the quantity $G$ (see \eqref{defG})
\begin{equation*}
 \max_{M_t} G = \max_{M_t}\{ \sigma_k^{2\alpha}\sum_{i<j}(\frac{1}{\lambda_i}-\frac{1}{\lambda_j})^2\}
\end{equation*}
 is non-increasing in time.
 \end{thm}

Before we prove Theorem \ref{thm3.1},  by combing Lemma \ref{lem3.1} and Theorem \ref{thm3.1}, we prove \eqref{3.1} and show that  if $\alpha\in[1/k, c_{\epsilon}(n,k) ]$, then
the strict convexity of the flow hypersurface is preserved under the flow \eqref{flow} for $\epsilon=0,1$. From Lemma \ref{lem3.1}, we obtain a uniform positive lower bound $C_0$
for the $\sigma_k$-curvature under the flow \eqref{flow}. By using Theorem \ref{thm3.1}, we know that there exists a constant $C_1$ which only depends on the
 initial hypersurface $M_0$ such that as long as the flow hypersurface is strictly convex, we have $\frac{G(p,t)}{n-1}\leq C_1$.
We denote $\frac{\lambda_{2}(p,t)}{\lambda_{1}(p,t)}$ by $r(p,t)$, since $M_t$ is axially symmetric and has two distinct principal curvatures $\lambda_1$ and $\lambda_2$ (with multiplicity $n-1$), we have
\begin{equation}\label{expressionG}
\begin{aligned}
\frac{G}{n-1}&=\sigma_k^{2\alpha}(\frac{1}{\lambda_1}-\frac{1}{\lambda_2})^2=\frac{\sigma_k^{2\alpha}}{\lambda_1\lambda_2}(\frac{\lambda_2}{\lambda_1}+\frac{\lambda_1}{\lambda_2}-2)\\
&=(\frac{\sigma_k^{2}}{(\lambda_1\lambda_2)^k})^{\frac{1}{k}}\sigma_k^{2(\alpha-\frac{1}{k})}(\frac{\lambda_2}{\lambda_1}+\frac{\lambda_1}{\lambda_2}-2)\\
&=(\binom{n-1}{k-1}+\binom{n-1}{k}r)^{\frac{2}{k}}r^{1-\frac{2}{k}}\sigma_k^{2(\alpha-\frac{1}{k})}(r+\frac{1}{r}-2)\leq C_1.
\end{aligned}
\end{equation}
Since $\sigma_k\geq C_0$, $\alpha\geq 1/k$, there exists a constant $C_2$ which only depends on $M_0$ such that
\begin{equation}\label{rbound}
(\binom{n-1}{k-1}+\binom{n-1}{k}r)^{\frac{2}{k}}r^{1-\frac{2}{k}}(r+\frac{1}{r}-2)\leq C_2.
\end{equation}
If $r\geq 1$, we obtain from \eqref{rbound} that ${\binom{n-1}{k}}^{\frac{2}{k}}(r-1)^2\leq C_2$, which implies that $r$ is bounded from above by a constant which only depends on $M_0$.
If $r<1$, we   obtain from \eqref{rbound} that ${\binom{n-1}{k-1}}^{\frac{2}{k}}(r^{1-\frac{2}{k}}(r-2)+r^{-\frac{2}{k}})\leq C_2$, which implies that $r$ is bounded from below by a positive constant which only depends on $M_0$. This means that  there exists some constant $C_3>1$ which only depends on $M_0$ such that
\begin{equation}\label{3.14}
C_3^{-1}\leq r(p,t)\leq C_3,
\end{equation}
which in combination with Lemma \ref{lem3.1} implies that the strict convexity of the flow hypersurface is preserved under the flow \eqref{flow}.
Moreover, we can obtain a uniform positive lower bound for the principal curvatures under the flow \eqref{flow} by combining \eqref{3.14} and Lemma \ref{lem3.1}.
Once we have the estimate \eqref{3.14}, we can obtain   immediately from \eqref{expressionG} that there exists some positive constant $C$ which only depends on $M_0$ such that
\begin{equation*}
r(p,t)+1/r(p,t)-2\le \frac{C}{\sigma_k^{2(\alpha-\frac{1}{k})}(p,t)},
\end{equation*}
which is equivalent to the key estimate \eqref{3.1}.

\begin{rem}
We note that when $\epsilon=0,1,~\alpha>0$, for the $\sigma_k^\alpha$-curvature flow \eqref{flow} with strictly convex initial hypersurfaces,  one can even obtain that the smallest principal curvature does not decrease along the flow by
applying Andrews' tensor maximum principle (see Theorem 3.2 in \cite{A2007}) to the evolution equation of the Weingarten tensor.  We refer to Theorem 5 of \cite{AMZ2013} in the case $\epsilon=0$ and Proposition 4.1 of  \cite{LILV2019} in the case $\epsilon=1$ for the details of the proof.
\end{rem}

\subsection{Proof of Theorem \ref{thm3.1}}\label{sec3.2}
In order to prove Theorem \ref{thm3.1}, we first give two important lemmas, which first appeared in \cite{AL2015}. For the readers' convenience, we give a brief proof here.
In the  proof of the following lemmas, at a given point $p\in M$, we  choose local coordinates $\xi^1,\cdots,\xi^n$ such that $g_{ij}=\delta_{ij},~\nabla_{\frac{\partial}{\partial \xi^i}}\frac{\partial}{\partial \xi^j}=0$ and $h^i_j=\text{diag}(\lambda_1,\cdots,\lambda_n)$ at $p$.
For convenience, we use notations $\dot{f}^i=\frac{\partial F}{\partial \lambda_i},~\dot{f}^{ij}=\frac{\partial^2F}{\partial \lambda_i\partial \lambda_j}$ and $\dot{g}^i=\frac{\partial G}{\partial \lambda_i},~\dot{g}^{ij}=\frac{\partial^2G}{\partial \lambda_i\partial \lambda_j}.$

\begin{lem}[\cite{AL2015}]\label{lem3.2}
	Let $F$ be a  smooth, symmetric and homogeneous of degree $a$ function of the principal curvatures $\lambda_1,\cdots,\lambda_n$ of a hypersurface $M$ in $\mathbb{R}^{n+1}(\epsilon)~(\epsilon=0,1)$, we define  $G$ by $G=F^2\sum_{i<j}(\frac{1}{\lambda_i}-\frac{1}{\lambda_j})^2$, then we have
	\begin{equation*}
	\dot{F}^{ml} h_{mr}h^r_l \dot{G}^{ij}h_{ij}+(1-a) F \dot{G}^{ij}h_{im} h^m_j = 0
	\end{equation*}
\end{lem}
\begin{proof}

By definition, $G$ is a symmetric and homogeneous of degree $ b =2(a-1)$ function of the principal curvatures. Using the Euler relation \eqref{2.2}, we have
\begin{equation*}
	\begin{aligned}
	&\dot{F}^{ml} h_{mr}h^r_l \dot{G}^{ij}h_{ij}+(1-a) F \dot{G}^{ij}h_{im} h^m_j\\
		=& b  G \dot{F}^{ml} h_{mr}h^r_l+(1-a) F \dot{G}^{ij}h_{im} h^m_j\\
		=&2(a-1)G\sum_i\dot{f}^i \lambda_i^2+(1-a) F\sum_i\dot{g}^i \lambda_i^2\\
		=&2(a-1)G\sum_i\dot{f}^i \lambda_i^2+(1-a) F\sum_i\Big(2F\dot{f}^i\sum_{k<j}(\frac{1}{\lambda_k}-\frac{1}{\lambda_j})^2+2F^2\sum_{j\neq i}(\frac{1}{\lambda_i}-\frac{1}{\lambda_j})(-\frac{1}{\lambda_i^2})\Big)\lambda_i^2\\
			=&2(a-1)F^3\sum_{i\neq j}(\frac{1}{\lambda_i}-\frac{1}{\lambda_j}) = 0.
	\end{aligned}
	\end{equation*}
\end{proof}
\begin{rem}
	In Lemma \ref{lem3.2}, we do not assume that $M$ is  axially symmetric.
\end{rem}

\begin{lem}[\cite{AL2015}]\label{lem3.4}
	Let $F$ and $G$ be two smooth, symmetric and homogeneous  functions of the principal curvatures $\lambda_1,\cdots,\lambda_n$ of a closed, strictly convex and  axially symmetric  hypersurface $M$  in $\mathbb{R}^{n+1}(\epsilon)~(\epsilon=0,1)$,  assume that $F$ is homogeneous of degree $a$ and $G$ is homogeneous of degree $ b $.
	At any stationary point of $G$, if $\dot{g}^1\neq 0$, $\lambda_1\neq \lambda_2=\cdots = \lambda_n$, then we have
	\begin{equation}\label{eqlem3.4}
	\begin{aligned}
	&(\dot{G}^{ij}\ddot{F}^{ml,rs}-\dot{F}^{ij}\ddot{G}^{ml,rs}) \nabla_i h_{ml} \nabla_j h_{rs}\\
	=&\{\frac{\dot{g}^1}{\lambda_2^2} a (a-1)F- \frac{\dot{f}^1}{\lambda_2^2}  b  ( b -1)G- \frac{2\dot{g}^1}{\lambda_2(\lambda_1-\lambda_2)}a F+\frac{2\dot{f}^1}{\lambda_2(\lambda_1-\lambda_2)} b  G \\
	&+(\frac{ b ^2 G^2}{\lambda_2^2 (\dot{g}^1)^2} - \frac{2  b  G \lambda_1}{\lambda_2^2 \dot{g}^1})(\dot{g}^1 \ddot{f}^{11} - \dot{f}^1 \ddot{g}^{11}) -2(n-1) \frac{ b  G}{\lambda_2 \dot{g}^1}(\dot{g}^1 \ddot{f}^{12} - \dot{f}^1 \ddot{g}^{12}) \} (\nabla_1 h_{22})^2.
	\end{aligned}
	\end{equation}
\end{lem}

\begin{proof}
Using \eqref{2.1}, the property of the axially symmetric hypersurfaces given in Section 2.3 (see \eqref{2.8}) and the Codazzi equations, we have
	\begin{equation}
\begin{aligned} \label{eq3}
&\dot{G}^{ij}\ddot{F}^{ml,rs} \nabla_i h_{ml} \nabla_j h_{rs}\\
=&\dot{g}^1\{ \ddot{f}^{11}(\nabla_1 h_{11} )^2+2(n-1) \ddot{f}^{12}\nabla_1 h_{11} \nabla_1 h_{22} + (n-1)\ddot{f}^{22} (\nabla_1 h_{22})^2 \\
&+ (n-1) (n-2) \ddot{f}^{23} (\nabla_1 h_{22})^2\} + 2(n-1) \dot{g}^2 \frac{\dot{f}^1-\dot{f}^2}{\lambda_1 - \lambda_2}(\nabla_1 h_{22})^2,
	\end{aligned}
\end{equation}
and
	\begin{equation}
\begin{aligned} \label{eq4}
&\dot{F}^{ij}\ddot{G}^{ml,rs} \nabla_i h_{ml} \nabla_j h_{rs}\\
=&\dot{f}^1\{ \ddot{g}^{11}(\nabla_1 h_{11} )^2+2(n-1) \ddot{g}^{12}\nabla_1 h_{11} \nabla_1 h_{22} + (n-1)\ddot{g}^{22} (\nabla_1 h_{22})^2\\
&+(n-1) (n-2) \ddot{g}^{23} (\nabla_1 h_{22})^2\} + 2(n-1) \dot{f}^2 \frac{\dot{g}^1-\dot{g}^2}{\lambda_1 - \lambda_2}(\nabla_1 h_{22})^2.
	\end{aligned}
\end{equation}
Since
\begin{equation*}\label{nablag}
\nabla_1 G= \dot{g}^1 \nabla_1 h_{11}+ (n-1) \dot{g}^2 \nabla_1 h_{22},
\end{equation*}
at any stationary point of $G$,
if $\dot{g}^1\neq 0$, then we have
\begin{equation}\label{nablah11}
 \nabla_1 h_{11} = \frac{1}{\dot{g}^1}(\nabla_1 G - (n-1) \dot{g}^2 \nabla_1 h_{22}).
\end{equation}
Substituting \eqref{nablah11} into \eqref{eq3} and \eqref{eq4}, we get
	\begin{equation}\label{eq5}
\begin{aligned}
&(\dot{G}^{ij}\ddot{F}^{ml,rs}-\dot{F}^{ij}\ddot{G}^{ml,rs}) \nabla_i h_{ml} \nabla_j h_{rs}\\\
=&(\nabla_1 G)^2 \cdot\frac{(\dot{g}^1 \ddot{f}^{11} - \dot{f}^1 \ddot{g}^{11})}{(\dot{g}^1)^2}\\
&+\nabla_1 h_{22} \nabla_1 G\cdot\frac{2(n-1)}{\dot{g}^1}\big (\dot{g}^1 \ddot{f}^{12} - \dot{f}^1 \ddot{g}^{12}-\frac{\dot{g}^2}{\dot{g}^1}(\dot{g}^1 \ddot{f}^{11} - \dot{f}^1 \ddot{g}^{11})\big) \\
&+(\nabla_1h_{22})^2\{(n-1)^2 (\frac{\dot{g}^2}{\dot{g}^1})^2(\dot{g}^1 \ddot{f}^{11} - \dot{f}^1 \ddot{g}^{11}) - 2(n-1)^2 \frac{\dot{g}^2}{\dot{g}^1}(\dot{g}^1 \ddot{f}^{12} - \dot{f}^1 \ddot{g}^{12}) \\
&+(n-1)(\dot{g}^1 \ddot{f}^{22} - \dot{f}^1 \ddot{g}^{22})+(n-1)(n-2)(\dot{g}^1 \ddot{f}^{23} - \dot{f}^1 \ddot{g}^{23})+2(n-1)\frac{\dot{g}^2 \dot{f}^1- \dot{f}^2 \dot{g}^1}{\lambda_1 -  \lambda_2}\}.
	\end{aligned}
\end{equation}
Using the Euler relation \eqref{2.2}, we obtain
	\begin{equation*}
\lambda_1 \dot{g}^1+(n-1) \lambda_2 \dot{g}^2 =  b  G,
\end{equation*}
which leads to the following relation
	\begin{equation}\label{g1g2}
\frac{\dot{g}^2}{\dot{g}^1} = -\frac{\lambda_1}{(n-1) \lambda_2}+\frac{ b  G}{(n-1) \lambda_2 \dot{g}^1}.
\end{equation}
Putting \eqref{g1g2} into \eqref{eq5}, we have that  at any stationary point of $G$, if $\dot{g}^1\neq 0$, then
\begin{equation*}\begin{aligned}
&(\dot{G}^{ij}\ddot{F}^{ml,rs}-\dot{F}^{ij}\ddot{G}^{ml,rs}) \nabla_i h_{ml} \nabla_j h_{rs}\\
=&\{ (\dot{g}^1 \ddot{f}^{11} - \dot{f}^1 \ddot{g}^{11})\big((\frac{\lambda_1}{\lambda_2})^2+ (\frac{ b  G}{\lambda_2 \dot{g}^1})^2 - \frac{2  b  G \lambda_1}{\lambda_2^2 \dot{g}^1}\big) \\
&+2(n-1)(\dot{g}^1 \ddot{f}^{12} - \dot{f}^1 \ddot{g}^{12}) (\frac{\lambda_1}{\lambda_2} - \frac{ b  G}{\lambda_2 \dot{g}^1})\\
&+(n-1)(\dot{g}^1 \ddot{f}^{22} - \dot{f}^1 \ddot{g}^{22})+(n-1)(n-2)(\dot{g}^1 \ddot{f}^{23} - \dot{f}^1 \ddot{g}^{23})\\
&+\frac{2\dot{f}^1 \dot{g}^1}{\lambda_1- \lambda_2}(-\frac{\lambda_1}{\lambda_2}+\frac{ b  G}{\lambda_2 \dot{g}^1}) - 2(n-1)\frac{ \dot{f}^2 \dot{g}^1}{\lambda_1- \lambda_2}\} (\nabla_1 h_{22})^2. \label{eq6}
\end{aligned}\end{equation*}
By the homogeneity of $F$ and $G$, we use the Euler relation again to get the following equations.
	\begin{equation*}
\begin{aligned}
\left\{
\begin{array}{ll}
\lambda_1 \dot{f}^1+(n-1) \lambda_2 \dot{f}^2 = a F\\
\lambda_1^2 \ddot{f}^{11}+2(n-1)\lambda_1 \lambda_2 \ddot{f}^{12}+(n-1)\lambda_2^2 \ddot{f}^{22}+(n-1)(n-2)\lambda_2^2 \ddot{f}^{23} = a (a - 1) F\\
\lambda_1^2 \ddot{g}^{11}+2(n-1)\lambda_1 \lambda_2 \ddot{g}^{12}+(n-1)\lambda_2^2 \ddot{g}^{22}+(n-1)(n-2)\lambda_2^2 \ddot{g}^{23} =  b  ( b  - 1) G
\end{array}
\right. .
	\end{aligned}
\end{equation*}
Thus \eqref{eq6}  can be simplified as follows.
	\begin{equation*}
\begin{aligned}
&(\dot{G}^{ij}\ddot{F}^{ml,rs}-\dot{F}^{ij}\ddot{G}^{ml,rs})\nabla_i h_{ml} \nabla_j h_{rs}\\
=&\{\frac{\dot{g}^1}{\lambda_2^2} a (a-1)F- \frac{\dot{f}^1}{\lambda_2^2}  b  ( b -1)G- \frac{2\dot{g}^1}{\lambda_2(\lambda_1-\lambda_2)}a F+\frac{2\dot{f}^1}{\lambda_2(\lambda_1-\lambda_2)} b  G \\
+&(\frac{ b ^2 G^2}{\lambda_2^2 (\dot{g}^1)^2} - \frac{2  b  G \lambda_1}{\lambda_2^2 \dot{g}^1})(\dot{g}^1 \ddot{f}^{11} - \dot{f}^1 \ddot{g}^{11}) -2(n-1) \frac{ b  G}{\lambda_2 \dot{g}^1}(\dot{g}^1 \ddot{f}^{12} - \dot{f}^1 \ddot{g}^{12}) \} (\nabla_1 h_{22})^2.
	\end{aligned}
\end{equation*}
This completes the proof of Lemma \ref{lem3.4}.
\end{proof}

Recall that the
	 evolution equation for $G$ (defined by \eqref{defG}) under the flow \eqref{flow} can be written in the following form (see \eqref{2.7}).
	\begin{equation}\label{evog}
	\begin{aligned}
\frac{\partial}{\partial t} G=&\mathcal{L} G+(\dot{G}^{ij}\ddot{F}^{ml,rs}-\dot{F}^{ij}\ddot{G}^{ml,rs})\nabla_ih_{ml}\nabla_jh_{rs}\\
&+\dot{F}^{ml}h_{mr}h^r_l\dot{G}^{ij}h_{ij}+(1-k\alpha)F\dot{G}^{ij}h_{im}h^m_j\\
&+\epsilon\big((1+k\alpha)F \dot{G}^{ij}g_{ij}-\dot{F}^{ml}g_{ml}\dot{G}^{ij}h_{ij}\big),
	\end{aligned}
	\end{equation}
	where $F=\sigma_k^{a}$  is the speed function  of the flow \eqref{flow} with degree $a=k\alpha$.
	We prove the two cases $\epsilon=0$ and $\epsilon=1$ separately.

\textbf{Case 1: $\epsilon=0$.} In order to apply the maximum principle, we need to show that the right-hand side of \eqref{evog} has a desired sign at stationary points of $G$.
	From the definition of $G$, we know that $G$ is homogeneous of degree $ b =2k \alpha-2$. In view of Lemma \ref{lem3.2}, we obtain that the zero-order terms of the right-hand side of \eqref{evog} are identically zero. In order to apply the maximum principle, it remains to prove that the gradient terms of the right-hand side of \eqref{evog} are non-positive at any maximum point of $G$. If $\lambda_1=\lambda_2$ at a maximum point $p\in M_{t_1}$, then we obtain that $G$ is identically $0$ on $M_{t_1}$, which means that $M_{t_1}$ is a round sphere, hence the right-hand side of  \eqref{evog} is identically $0$ at any point $p\in M_{t_1}$ and $M_t$ is a round sphere for any $t\geq t_1$. Note that the gradient terms $(\dot{G}^{ij}\ddot{F}^{ml,rs}-\dot{F}^{ij}\ddot{G}^{ml,rs})\nabla_ih_{ml}\nabla_jh_{rs}$ can be regarded as a function of
$\lambda_1,\lambda_2, n,k,\alpha,\nabla_1h_{11},\nabla_1h_{22}$, the set $\{(\lambda_1,\lambda_2, n,k,\alpha,\nabla_1h_{11},\nabla_1h_{22})|\nabla_1G=0,\lambda_1\neq\lambda_2,\dot{g}^1\ne0\}$ is a dense subset of $\{(\lambda_1,\lambda_2, n,k,\alpha,\nabla_1h_{11},\nabla_1h_{22})|\nabla_1G=0,\lambda_1\neq\lambda_2\}$, by the property of continuity, we only need to prove that $(\dot{G}^{ij}\ddot{F}^{ml,rs}-\dot{F}^{ij}\ddot{G}^{ml,rs})\nabla_ih_{ml}\nabla_jh_{rs}$ are non-positive  in the case which satisfies that $\lambda_1\neq \lambda_2$ and $\dot{g}^1\neq0$, and we can apply  Lemma \ref{lem3.4} to simplify the gradient terms in \eqref{evog}.
	
	For convenience, we  use the following notations:
\begin{equation*}
\phi=\frac{\lambda_2-\lambda_1}{\lambda_1 \lambda_2},~
\Phi=\sum_{i<j}(\frac{1}{\lambda_i}-\frac{1}{\lambda_j})^2,
\end{equation*}
then $G=\sigma_k^{2\alpha  } \Phi $ and we can compute the derivatives of $F, ~G$ and $\Phi$ as follows:
\begin{equation*}
\begin{aligned}
\dot{f}^r&=\alpha \sigma_k^{\alpha -1} \dot{\sigma}_k^r, ~\ddot{f}^{rr}=\alpha(\alpha-1)\sigma_k^{\alpha-2}(\dot{\sigma}_k^r)^2, \\
\ddot{f}^{rs}&=\alpha(\alpha-1)\sigma_k^{\alpha-2}\dot{\sigma}_k^r \dot{\sigma}_k^s+ \alpha \sigma_k^{\alpha-1}\ddot{\sigma}_k^{rs},~r\ne s,\\
\dot{g}^r&=2\alpha \sigma_k^{2\alpha-1}\dot{\sigma}_k^r\Phi+\sigma_k^{2\alpha} \dot{\Phi}^r , \\
\ddot{g}^{rr}&=2\alpha(2\alpha-1)\sigma_k^{2\alpha-2}(\dot{\sigma}_k^r)^2\Phi+ 4\alpha \sigma_k^{2\alpha-1}\dot{\sigma}_k^r\dot{\Phi}^r+\sigma_k^{2\alpha}\ddot{\Phi}^{rr},
\end{aligned}
\end{equation*}
\begin{equation*}
\begin{aligned}
\ddot{g}^{rs}&=2\alpha(2\alpha-1)\sigma_k^{2\alpha-2}\dot{\sigma}_k^r\dot{\sigma}_k^s\Phi+2\alpha \sigma_k^{2\alpha-1}\ddot{\sigma}_k^{rs}\Phi+2\alpha \sigma_k^{2\alpha-1}\dot{\sigma}_k^r\dot{\Phi}^s\\
&~~~+2\alpha \sigma_k^{2\alpha-1}\dot{\sigma}_k^s\dot{\Phi}^r+\sigma_k^{2\alpha}\ddot{\Phi}^{rs} ,~r\neq s,\\
\dot{\Phi}^r &= -2\lambda_r^{-2} \sum_{i \ne r}(\frac{1}{\lambda_r}-\frac{1}{\lambda_i}),\\
\ddot{\Phi}^{rr}&=\frac{\partial}{\partial \lambda_r} \dot{\Phi}^r =2(n-1)\lambda_r^{-4}+4\lambda_r^{-3} \sum_{i \ne r}(\frac{1}{\lambda_r}-\frac{1}{\lambda_i}),\\
~\ddot{\Phi}^{rs}&=\frac{\partial}{\partial \lambda_s} \dot{\Phi}^r =-2\lambda_r^{-2}\lambda_s^{-2},~r\neq s.
\end{aligned}
\end{equation*}

Since $M_t$ is axially symmetric and  we have that $\lambda_2=\cdots=\lambda_n$, it follows that at any point $p \in M_t$, we have
\begin{equation*}
\begin{aligned}
\sigma_k =&\binom{n-1}{k-1}\lambda_1\lambda_2^{k-1}+\binom{n-1}{k}\lambda_2^k,~1\leq k\leq n-1;~\sigma_n=\lambda_1\lambda_2^{n-1},\\
\dot{\sigma}_k^1=&\binom{n-1}{k-1}\lambda_2^{k-1}=\frac{k}{\lambda_1 k+\lambda_2 (n-k)}\sigma_k ,\\
\dot{\sigma}_k^2=&\binom{n-2}{k-2}\lambda_1\lambda_2^{k-2}+\binom{n-2}{k-1}\lambda_2^{k-1}=\frac{k (\lambda_1 (k-1)+\lambda_2 (n-k))}{\lambda_2 (n-1) (\lambda_1 k+\lambda_2 (n-k))}\sigma_k ,\\
\ddot{\sigma}_k^{12}=&\binom{n-2}{k-2}\lambda_2^{k-2}=\frac{k(k-1) }{\lambda_2 (n-1) (\lambda_1 k+\lambda_2 (n-k))}\sigma_k ,\\
\dot{f}^1=&\frac{\alpha k}{\lambda_1 k+ \lambda_2 (n-k)} \sigma_k^{\alpha},~\dot{f}^2=\frac{\alpha k(\lambda_1(k-1)+\lambda_2(n-k))}{\lambda_2 (n-1) (\lambda_1 k+ \lambda_2 (n-k))}
\sigma_k^{\alpha},\\
\ddot{f}^{11}=&\frac{\alpha (\alpha-1) k^2}{(\lambda_1 k+ \lambda_2 (n-k))^2} \sigma_k^{\alpha},~\ddot{f}^{12}=\frac{\alpha^2 k^2 (k-1)\lambda_1+\alpha k (\alpha k-1)(n-k) \lambda_2}{\lambda_2 (n-1)(\lambda_1 k+ \lambda_2 (n-k))^2} \sigma_k^{\alpha},\\
\dot{\Phi}^1=&-2(n-1)\lambda_1^{-2} \phi,~\dot{\Phi}^2=2\lambda_2^{-2} \phi,\\
\ddot{\Phi}^{11}=&2(n-1)\lambda_1^{-4}+4(n-1) \lambda_1^{-3}\phi,~\ddot{\Phi}^{12}=-2\lambda_1^{-2} \lambda_2^{-2},\\
\dot{g}^1=&\frac{2(n-1)}{\lambda_1^3 \lambda_2^2(\lambda_1 k+ \lambda_2 (n-k))}\sigma_k^{2\alpha}(\lambda_1-\lambda_2)(\alpha k \lambda_1^2-(\alpha-1)k \lambda_1\lambda_2+(n-k)\lambda_2^2),\\
\ddot{g}^{11}=&\frac{2(n-1)}{\lambda_1^{4} \lambda_2^{2} (\lambda_1 k+ \lambda_2(n-k))^{2}} \sigma_k^{2\alpha}\{\lambda_1^4 \cdot \alpha k^2(2\alpha-1)- \lambda_1^3 \lambda_2\cdot 2 k^2(\alpha-1)(2\alpha-1)\\
&+\lambda_1^2\lambda_2^2\cdot k(\alpha-1)(2\alpha k-7k+4n)-\lambda_1 \lambda_2^3\cdot 2 (n-k)(2\alpha k-4k+n)\\
&+\lambda_2^4\cdot 3(n-k)^2\},
\end{aligned}
\end{equation*}
and
\begin{equation*}
\begin{aligned}
\ddot{g}^{12}=&
\frac{2}{\lambda_1^{3}\lambda_2^{3}(\lambda_1 k+\lambda_2(n-k))^{2}} \sigma_k^{2\alpha}\{\lambda_1^4 \cdot 2 \alpha k^2 (\alpha (k-1)-1)\\
&+ \lambda_1^3\lambda_2 \cdot k (2\alpha k^2(1-3\alpha)+k(2n\alpha^2+4\alpha^2+3\alpha-1)-3n\alpha)\\
&+\lambda_1^2 \lambda_2^2\cdot 2 k (\alpha-1)  (3 \alpha k^2-k (2 \alpha n+\alpha+1)+n)\\
&+\lambda_1\lambda_2^3 \cdot (n-k) ( 2 (\alpha-3) \alpha k^2+k(2 \alpha n+\alpha+1)-n)\\
&+ \lambda_2^4 \cdot(-2 \alpha k (n-k)^2)
\}.
\end{aligned}
\end{equation*}

At any stationary point $p \in M_t$, by substituting the derivatives of $F$ and $G$ into the right-hand side of  \eqref{eqlem3.4} and noting that $F$ has degree $\alpha k$, we have
\begin{equation*}
\begin{aligned}
\frac{\dot{g}^1}{\lambda_2^2} \alpha k (\alpha k-1)F=&\frac{2 \alpha k (n-1) (\alpha k-1) (\lambda_1-\lambda_2) \left(\lambda_1^2 \alpha k-\lambda_1  \lambda_2 (\alpha-1) k+\lambda_2^2 (n-k)\right)}{\lambda_1^3 \lambda_2^4 (k\lambda_1+\lambda_2 (n-k))}\sigma_k^{3\alpha},\\
 \frac{\dot{f}^1}{\lambda_2^2}  b  ( b -1)G=&\frac{2 \alpha k (n-1)(\alpha k-1) (2 \alpha k-3) (\lambda_1-\lambda_2)^2 }{\lambda_1^2 \lambda_2^4 (\lambda_1 k+\lambda_2 (n-k))}\sigma_k^{3\alpha},\\
 \frac{2\dot{g}^1}{\lambda_2(\lambda_1-\lambda_2)}\alpha k F=&\frac{4 \alpha k (n-1) \left(\lambda_1^2 \alpha k-\lambda_1  \lambda_2(\alpha-1) k+\lambda_2^2 (n-k)\right)}{\lambda_1^3 \lambda_2^3 (\lambda_1 k+\lambda_2 (n-k))}\sigma_k^{3\alpha},\\
\frac{2\dot{f}^1}{\lambda_2(\lambda_1-\lambda_2)} b  G=&\frac{4 \alpha k (n-1)  (\alpha k-1)(\lambda_1-\lambda_2)}{\lambda_1^2 \lambda_2^3 (\lambda_1 k+\lambda_2 (n-k))}\sigma_k^{3\alpha},
\end{aligned}
\end{equation*}
\begin{equation*}
\begin{aligned}
&\quad(\frac{ b ^2 G^2}{\lambda_2^2 (\dot{g}^1)^2} - \frac{2  b  G \lambda_1}{\lambda_2^2 \dot{g}^1})(\dot{g}^1 \ddot{f}^{11} - \dot{f}^1 \ddot{g}^{11})\\
=&\frac{-2 \alpha k (n-1) (\alpha k-1)(\lambda_1-\lambda_2) }{\lambda_1^{2} \lambda_2^{4} (\lambda_1 k+\lambda_2 (n-k))^{2}(\lambda_1^2 \alpha k-\lambda_1 \lambda_2(\alpha-1)k+\lambda_2^2(n-k))^{2}}\sigma_k^{3\alpha}\cdot\\
& \big(\lambda_1^4 \alpha^2 k^2-\lambda_1^3 \lambda_2 (2 \alpha^2-3 \alpha+1)  k^2+\lambda_1^2  \lambda_2^2 (\alpha-1) k ((\alpha-5) k+3 n)\\
&\quad+\lambda_1 \lambda_2^3 (k-n) ((3 \alpha-7) k+2 n)+3 \lambda_2^4 (n-k)^2\big)\cdot\\
& \big(\lambda_1^2 k (\alpha (k-2)-1)+\lambda_1 \lambda_2 (\alpha k (-2 k+n+2)-n)+\lambda_2^2 (\alpha k+1) (k-n)\big),
\end{aligned}
\end{equation*}
and
\begin{equation*}
\begin{aligned}
& \quad2(n-1) \frac{ b  G}{\lambda_2 \dot{g}^1}(\dot{g}^1 \ddot{f}^{12} - \dot{f}^1 \ddot{g}^{12})\\
=&\frac{-4 \alpha k (n-1) (\alpha k-1)(\lambda_1-\lambda_2) }{\lambda_1^{2} \lambda_2^{4} (\lambda_1 k+\lambda_2 (n-k))^{2}(\lambda_1^2 \alpha k-\lambda_1 \lambda_2(\alpha-1)k+\lambda_2^2(n-k))}\sigma_k^{3\alpha} \cdot\\
 &\Big(\lambda_1^4 \alpha k^2 (\alpha (k-1)-2)+\lambda_1^3 \lambda_2 k ((\alpha-3 \alpha^2) k^2+k (\alpha^2 (n+2)+3 \alpha-1)-2 \alpha n)\\
 &\quad+\lambda_1^2  \lambda_2^2 (\alpha-1) k (3 \alpha k^2-k (2 \alpha n+\alpha+1)+n)\\
 &\quad+\lambda_1 \lambda_2^3 k (n-k) (k\alpha(\alpha-3)+(n+1)\alpha-1)-\lambda_2^4 (\alpha k+1) (n-k)^2\Big).
\end{aligned}
\end{equation*}
Using the above formulas, we obtain that
\begin{equation*}
\begin{aligned}
&\frac{\dot{g}^1}{\lambda_2^2} \alpha k (\alpha k-1)F- \frac{\dot{f}^1}{\lambda_2^2}  b  ( b -1)G- \frac{2\dot{g}^1}{\lambda_2(\lambda_1-\lambda_2)}\alpha k F+\frac{2\dot{f}^1}{\lambda_2(\lambda_1-\lambda_2)} b  G \\
&+(\frac{ b ^2 G^2}{\lambda_2^2 (\dot{g}^1)^2} - \frac{2  b  G \lambda_1}{\lambda_2^2 \dot{g}^1})(\dot{g}^1 \ddot{f}^{11} - \dot{f}^1 \ddot{g}^{11}) -2(n-1) \frac{ b  G}{\lambda_2 \dot{g}^1}(\dot{g}^1 \ddot{f}^{12} - \dot{f}^1 \ddot{g}^{12})\\
=&\frac{2(n-1) \alpha k \sigma_k^{3\alpha}}{\lambda_1^3 \lambda_2^3(\lambda_1 k+ \lambda_2(n-k))(\lambda_1^2 \alpha k-\lambda_1 \lambda_2(\alpha-1)k+\lambda_2^2(n-k))^2}\cdot P,
\end{aligned}
\end{equation*}
where $P$ is a polynomial given by
\begin{equation*}
\begin{aligned}
P:=&\quad\lambda_1^6 k^2 \big(\alpha (k-1)-1\big) \big(\alpha (k+2)-1\big)\\
&+\lambda_1^5 \lambda_2 k \big(\alpha^2 k(-4k^2+3k(n-2)+n+6)+\alpha(10k^2-6k n-n)+3n-6k\big)\\
&+\lambda_1^4 \lambda_2^2 \big(\alpha^2k^2(6k^2+3k(4-3n)+2n^2-5n-6)+\alpha k(k^3-24k^2+2k(11n+3)-n(4n+1))\\
&\quad-k^3+12k^2-13k n+2n^2\big)\\
&+\lambda_1^3 \lambda_2^3 \big(\alpha^2k^2(-4k^2+k(9n-10)-4n^2+7n+2)\\
&\quad+\alpha k(-4k^3+k^2(3n+32)-2k(19n+4)+5n(2n+1)) +2k^3-k^2(3n+10)+17k n-6n^2\big)\\
&+\lambda_1^2 \lambda_2^4 (k-n)\big(\alpha^2 k^2(k-2n+3)+\alpha k(6k^2-k(3n+22)+12n+3)+3k(n+1)-4n\big)\\
&-\lambda_1 \lambda_2^5 (n-k)^2 (\alpha k(4k-n-6)+2k+n)\\
&+\lambda_2^6 (\alpha k+1) (k-n)^3.
\end{aligned}
\end{equation*}
It follows that at any stationary point $p$ of $G$, if $\dot{g}^1\neq 0$, from Lemma \ref{lem3.4}, we have
\begin{equation*}
\begin{aligned}
&(\dot{G}^{ij}\ddot{F}^{ml,rs}-\dot{F}^{ij}\ddot{G}^{ml,rs})\nabla_i h_{ml} \nabla_j h_{rs} \\
=&\frac{2(n-1) \alpha k \sigma_k^{3\alpha}}{\lambda_1^3 \lambda_2^3(\lambda_1^2 \alpha k-\lambda_1 \lambda_2(\alpha-1)k+\lambda_2^2(n-k))^2(\lambda_1 k+ \lambda_2(n-k))}(\nabla_1 h_{22})^2\cdot P.
\end{aligned}
\end{equation*}
If   $\nabla_1 h_{22}=0$ at $p$, then $(\dot{G}^{ij}\ddot{F}^{ml,rs}-\dot{F}^{ij}\ddot{G}^{ml,rs})\nabla_i h_{ml} \nabla_j h_{rs}=0$ at $p$.
If $\nabla_1 h_{22}\neq0$ at $p$, let $x=\frac{\lambda_1}{\lambda_2}$ and  we define a polynomial $Q$ by
\begin{equation}\begin{aligned}\label{Q}
Q:=&\frac{P}{\lambda_2^6}\\
=&\quad x^6 k^2 (\alpha (k-1)-1) (\alpha (k+2)-1)\\
&+x^5 k \big(\alpha^2 k(-4k^2+3k(n-2)+n+6)+\alpha(10k^2-6k n-n)+3n-6k\big)\\
&+x^4\big(\alpha^2k^2(6k^2+3k(4-3n)+2n^2-5n-6)+\alpha k(k^3-24k^2+2k(11n+3)-n(4n+1))\\
&\quad-k^3+12k^2-13k n+2n^2\big)\\
&+x^3  \big(\alpha^2k^2(-4k^2+k(9n-10)-4n^2+7n+2)\\
&\quad+\alpha k(-4k^3+k^2(3n+32)-2k(19n+4)+5n(2n+1)) +2k^3-k^2(3n+10)+17k n-6n^2\big)\\
&+x^2(k-n)\big(\alpha^2 k^2(k-2n+3)+\alpha k(6k^2-k(3n+22)+12n+3)+3k(n+1)-4n\big)\\
&+x (n-k)^2 (-\alpha k(4k-n-6)-2k-n)\\
&+ (\alpha k+1) (k-n)^3.
\end{aligned}\end{equation}
In view of the above relations, in order to prove that the gradient terms are non-positive, it remains to find out for which $\{k, n, \alpha\}$, $Q=Q(x,k,n,\alpha)$ is non-positive for any $x>0$.

Note that $Q(x,k,n,\alpha)$ can also be regarded as a quadratic polynomial of $\alpha$,  and the coefficient of $\alpha^2$ is given by
$$
k^2 x^2 (x-1)^2 \Big((k^2+k-2)x^2+(n(3k+1)-2k^2-4k+2)x+(n-k)(2n-k-3)\Big).
$$
It is obvious that when $1\leq k\leq n-1$, we have $k^2+k-2\geq0$, $n(3k+1)-2k^2-4k+2\geq (k+1)(3k+1)-2k^2-4k+2=k^2+3>0$ and $(n-k)(2n-k-3)\geq 0$.
When $k=n$, we have $(k^2+k-2)x^2+(n(3k+1)-2k^2-4k+2)x+(n-k)(2n-k-3)=(n^2+n-2)x^2+(n-1)(n-2)x\ge 0.$
Hence, $Q(x,k,n,\alpha)$ is a convex function in $\alpha$ for all $n \geq 3 $, $1\le k\le n$ and $x\ge 0$, which means that in order to prove that there exists a constant $ c_0(n,k) >\frac{1}{k}$ such that $Q$ is non-positive for any $x>0$ and $\alpha\in[1/k, c_0(n,k) ]$, we only need to prove that for each $n$ and any fixed $k$ with $1\leq k\leq n$, there exists a constant $ c_0(n,k) >\frac{1}{k}$ such that
\begin{itemize}
\item [(i)] $Q$ is non-positive for any $x>0$ and $\alpha=1/k$.
\item [(ii)] $Q$ is non-positive for any $x>0$ and $\alpha= c_0(n,k) $.
\end{itemize}
The conclusion in (i) is trivial:
$$
Q(x,k,n,1/k)=-2 (n+(x-1) (k+x))^3<0,~\forall~x>0,~1\leq k\leq n.
$$

In the case $1<k\le n \le k^2$, we will prove that $ c_0(n,k) =\frac{1}{k-1}$ satisfies the conclusion in (ii).
For general case, by applying Sturm's theorem, we will prove that for each $n$ and any fixed $k$ with $1\leq k\leq n$, there exists a constant $ c_0(n,k) >\frac{1}{k}$ such that the conclusion in (ii) holds. We can also prove that $ c_0(n,1)\ge 1+\frac{7}{n}$ and $ c_0(n,k) \geq \frac{1}{k}+\frac{k}{(k-1)n}$ for $k\ge 2$ and $n\geq k^2$. Since the proof is long and technical, we will give the proof in the appendix, see Propositions  \ref{thmA1}-\ref{thmA4}.

Once we have obtained that for each $n$ and any fixed $k$ with $1\leq k\leq n$, there exists a constant $ c_0(n,k) >\frac{1}{k}$ such that $Q(x,k,n,\alpha)$ is non-positive for any $x>0$ and $\alpha\in[1/k, c_0(n,k) ]$, we can apply the maximum principle directly to conclude that
 the maximum of the quantity $G$
is non-increasing in time.

\textbf{Case 2: $\epsilon=1$.} In order to estimate the zero-order terms, we first prove the following lemma.
\begin{lem}\label{lemkappa1}
Under the same assumption of Theorem \ref{thm3.1}, for the curvature functions $F=\sigma_k^\alpha$  and  $G=F^2\sum_{i<j}(\frac{1}{\lambda_i}-\frac{1}{\lambda_j})^2$, we have
	\begin{equation}\label{Q2}
\begin{aligned}
	&\epsilon\big((1+k\alpha)F \dot{G}^{ij}g_{ij}-\dot{F}^{ml}g_{ml}\dot{G}^{ij}h_{ij}\big)\\
=&\frac{2\epsilon F G}{\lambda_1\lambda_2(\lambda_1 k+\lambda_2(n-k))}\cdot\\
&\Big(k((k-2)\alpha-1)\lambda_1^2-(k\alpha(2k-n-2)+n)\lambda_1\lambda_2-(n-k)(1+k\alpha)\lambda_2^2\Big).
\end{aligned}
	\end{equation}
\end{lem}
\begin{proof}
Using the formulas in the proof of Case 1 for $\epsilon=0$, we have
\begin{equation}\label{gi}
\dot{G}^{ij}g_{ij}=\dot{g}^1+\cdots+\dot{g}^n=2 F\Phi(\dot{f}^1+\cdots+\dot{f}^n)-2F^2\Phi(\frac{1}{\lambda_1}+\frac{1}{\lambda_2}),
\end{equation}
and
\begin{equation}\label{fi}
\dot{F}^{ij}g_{ij}=\dot{f}^1+\cdots+\dot{f}^n=\frac{k\alpha F(\lambda_1(k-1)+\lambda_2(n+1-k))}{\lambda_2(\lambda_1 k+\lambda_2(n-k))}.
\end{equation}
On the other hand, since $G$ is homogeneous of degree $2(k\alpha-1)$, we have
\begin{equation}\label{g0}
\dot{G}^{ij}h_{ij}=2(k\alpha-1) G.
\end{equation}
\eqref{Q2} follows immediately from \eqref{gi}, \eqref{fi} and \eqref{g0}.
\end{proof}

For any $n\geq 3$ and fixed $k$ with $1\leq k\leq n$, we define
\begin{equation}\label{c2nk}
c_2(n,k)=\left\{
\begin{aligned}
 & \frac{4\sqrt{n(n-k)(n+2-2k)}-2k(n+2)+n(6+n)}{k(n-2)^2}, \\
 &\quad~\text{if}~ k=1,2,~n\geq3,~\text{or}~k\geq 3,~n>k(k-1); \\
 & \frac{1}{k-2},~\text{otherwise}.
\end{aligned}
\right.
\end{equation}
We have the following claim:
\begin{cla}\label{cla1}
If $\alpha\in[1/k, c_2(n,k)]$, then for all $\lambda_1>0,\lambda_2>0$, we have
\begin{equation*}
k((k-2)\alpha-1)\lambda_1^2-(k\alpha(2k-n-2)+n)\lambda_1\lambda_2-(n-k)(1+k\alpha)\lambda_2^2\leq 0.
\end{equation*}
\end{cla}
\begin{proof}
We discuss the following two cases.

(i) In the case either $k=1,2,~n\geq3$, or $k\geq 3,~n>k(k-1)$, if
$$\frac{1}{k}\leq\alpha\leq
\frac{4\sqrt{n(n-k)(n+2-2k)}-2k(n+2)+n(6+n)}{k(n-2)^2},$$
then the discriminant of the  quadratic equation of $\lambda_2$
\begin{equation*}
k((k-2)\alpha-1)\lambda_1^2-(k\alpha(2k-n-2)+n)\lambda_1\lambda_2-(n-k)(1+k\alpha)\lambda_2^2=0
\end{equation*}
is non-positive, which implies that
\begin{equation*}
k((k-2)\alpha-1)\lambda_1^2-(k\alpha(2k-n-2)+n)\lambda_1\lambda_2-(n-k)(1+k\alpha)\lambda_2^2
\end{equation*}
is non-positive, since the coefficient of $\lambda_2^2$ is negative.

(ii) In the remaining cases, i.e., either (a) $k\geq 3,~n+2\leq 2k$, or (b) $k\geq 3,~n+2>2k,~n\leq k(k-1)$,
if $\frac{1}{k}\leq \alpha\leq\frac{1}{k-2}$, we will prove that all the coefficients of the quadratic polynomial
\begin{equation*}
k((k-2)\alpha-1)\lambda_1^2-(k\alpha(2k-n-2)+n)\lambda_1\lambda_2-(n-k)(1+k\alpha)\lambda_2^2
\end{equation*}
are non-positive.
(a) When $k\geq 3,~n+2\leq 2k$, $\frac{1}{k}\leq\alpha\leq\frac{1}{k-2}$, the above conclusion is obvious.
(b) When $k\geq 3,~n+2>2k,~n\leq k(k-1)$, $\frac{1}{k}\leq \alpha\leq\frac{1}{k-2}$, then we have
\begin{equation*}
-(k\alpha(2k-n-2)+n)\leq -(k\cdot\frac{1}{k-2}\cdot(2k-n-2)+n)=\frac{2(n-k(k-1))}{k-2}\leq0.
\end{equation*}
\end{proof}
Finally, in order to apply the maximum principle, we need to show that the right-hand side of \eqref{evog} has a desired sign at stationary points of $G$. In view of Lemma \ref{lem3.2}, Lemma \ref{lemkappa1} and Claim \ref{cla1}, we obtain that for any $n \geq 3$ and fixed $k$ with $1\leq k\leq n$, if $\alpha\in[1/k, c_2(n,k)]$, where $c_2(n,k)$ is defined by \eqref{c2nk}, then the zero-order terms of the right-hand side of \eqref{evog} are non-positive. From the proof of Case 1 for $\epsilon=0$, we obtain that the gradient terms are non-positive if $\alpha\in[1/k, c_0(n,k) ]$.
Therefore, if $\alpha\in[1/k, c_1(n,k) ]$, with $c_1(n,k)=\min\{c_0(n,k),~c_2(n,k)\}$, then the
right-hand side of \eqref{evog} is non-positive at stationary points of $G$,
and we can apply the maximum principle directly to conclude that
 the maximum of the quantity $G$
is non-increasing in time.

\begin{rem}
In the appendix, we will prove that  $c_0(3,1)=3.64...,~ c_0(4,1)=2.93...$ (see \eqref{listc0}), since $c_1(n,k)=\min\{c_0(n,k),~c_2(n,k)\}$, from the expression of
$c_2(n,k)$ given by \eqref{c2nk},  we can easily obtain that $c_1(3,1)=3.64...,~ c_1(4,1)=2.93...$.
We prove in the appendix that in the case $1<k\le n \le k^2$, we can choose $c_0(n,k)=\frac{1}{k-1}$.
We also show that $ c_0(n,1)\ge 1+\frac{7}{n}$ and $ c_0(n,k) \geq \frac{1}{k}+\frac{k}{(k-1)n}$ for $k\ge 2$ and $n\geq k^2$.
Since $c_2(n,k)$ is written explicitly, by elementary calculation, we can prove directly that   $c_2(n,k)\geq \frac{1}{k-1}$ in the case $1<k\le n \le k^2$,
$ c_2(n,1)\ge 1+\frac{7}{n}$ and $ c_2(n,k) \geq \frac{1}{k}+\frac{k}{(k-1)n}$ for $k\ge 2$ and $n\geq k^2$. As $c_1(n,k)=\min\{c_0(n,k),~c_2(n,k)\}$, we conclude that
we can choose $c_1(n,k)=\frac{1}{k-1}$ in the case $1<k\le n \le k^2$, $ c_1(n,1)\ge 1+\frac{7}{n}$,   and $ c_1(n,k) \geq \frac{1}{k}+\frac{k}{(k-1)n}$ for $k\ge 2$ and $n\geq k^2$.
\end{rem}

\section{Proof of Theorems \ref{thm1.1}-\ref{thm1.2}}\label{sec4}
\subsection{Contraction to a point}\label{sec4.1}
When $\epsilon=0,1$, any closed convex hypersurface in $\mathbb{R}^{n+1}(\epsilon)$ bounds a convex body in $\mathbb{R}^{n+1}(\epsilon)$. We define the
 inner radius $\rho_-(t)$ and the outer radius $\rho_+(t)$ as follows.
\begin{equation*}
\begin{aligned}
\rho_-(t) =& \sup\{r:B_r(y) \textrm{ is enclosed by } \hat{M}_t ~ \textrm{for some}~ y\in\mathbb{R}^{n+1}(\epsilon)\},\\
\rho_+(t) =& \inf\{r:B_r(y) \textrm{ encloses } \hat{M}_t  ~\textrm{for some}~ y\in\mathbb{R}^{n+1}(\epsilon)\},
\end{aligned}
\end{equation*}
where $\hat{M}_t$ is the convex body enclosed by $M_t$, $B_r(y)$ is a geodesic ball in $\mathbb{R}^{n+1}(\epsilon)$ with center $y\in\mathbb{R}^{n+1}(\epsilon)$, and $r$
is the radius of the geodesic ball $B_r(y)$ in geodesic polar coordinates. It follows from the scalar parabolic equation \eqref{s2:ut-evl} that the convex bodies $\hat{M}_t$ satisfies that
\begin{equation*}
t_1<t_2~\Longrightarrow~\hat{M}_{t_2}\subset\hat{M}_{t_1}.
\end{equation*}
Recall that we obtained some pinching estimates in Section 3, see \eqref{3.1} and \eqref{3.14}. In particular, by using the pinching estimate \eqref{3.14}, after an analogous
argument to that in \cite[\S 3]{SCHULZE2006} (for the case $\epsilon=0$) and \cite[\S 6]{Gerhardt2015} (for the case $\epsilon=1$), we can obtain that there exists
a positive constant $C_4$ which only depends on $M_0$ such that
\begin{equation}\label{pinchingrho}
\rho_+(t)\leq C_4\rho_-(t), ~\textrm{for all} ~t\in[0,T) ~\textrm{when}~ \epsilon=0, \textrm{ and for} ~t\in[t_\delta,T)~ \textrm{when}~ \epsilon=1,
\end{equation}
where $t_\delta=T-\delta$ and $\delta$ is sufficiently small. Then we can apply a technique of Tso \cite{Tso1985} to prove that the $\sigma_k$-curvature remains bounded from above (see Lemma \ref{upper0} and Lemma \ref{upper1} below) as long as
the flow \eqref{flow} bounds a non-vanishing volume, which together with the pinching estimate \eqref{3.14} implies a uniform upper bound for the principal curvatures.
Since the flow hypersurface $M_t$ is also uniformly strictly convex, we obtain that the flow \eqref{flow} remains to be uniformly parabolic.
Since the speed function can be written in the form $F=\sigma_k^{\alpha}=(\sigma_k^{1/k})^{k\alpha}$ and $\sigma_k^{1/k}$ is a concave function of the principal curvatures,
we can apply the  H\"{o}lder estimate by Andrews \cite[Theorem 6]{andrews2004fully} (we can also apply the H\"{o}lder estimate in the case of one space dimension in  \cite{lieberman1996second}, since
axially symmetric hypersurface can be written
as a graph on the unit sphere in geodesic polar coordinates and the graph function  has only one space variable) and
parabolic Schauder estimate \cite{lieberman1996second} to get uniform $C^\infty$ estimates of the solution, hence the solution can be extended beyond $T$, which contradicts the maximality of $T$.
This means that the inner radius $\rho_-(t)$ converges to $0$ as $t\to T$. It follows from \eqref{pinchingrho}  that the outer radius $\rho_+(t)$ also converges to $0$ as $t\to T$.
Therefore, the flow hypersurfaces remain smooth until they shrink to a point.
\begin{rem}
We note that by using the Gauss map parametrization, Andrews, McCoy and Zheng \cite{AMZ2013} proved that for contracting flow of strictly convex hypersurfaces in Euclidean space with the speed function $F^\alpha$ satisfying that $\alpha>0$, $F$ is homogeneous of degree one, the dual function of $F$ is concave and approaches zero on the
boundary of the positive cone, the flow hypersurfaces will shrink to a point as $t$ approaches the maximal time $T$. Li and Lv \cite{LILV2019} obtained similar results for the
contracting flow in the sphere.
\end{rem}

\subsection{Rescaling and convergence  for the flow \eqref{flow} with $\epsilon=0$}\label{sec4.2}
In this subsection, we prove the last part of Theorem \ref{thm1.1} by adapting the arguments of F. Schulze in  \cite{SCHULZE2006} for $k=1$ with an initial curvature pinching condition.
As remarked in Remark \ref{rem1.3}, we will consider the case $\alpha\in(1/k,c_0(n,k)]$.
The key step is that we can write the evolution equation of $\tilde{\sigma}_k$-curvature of the rescaled hypersurface in a special form such that we can
apply the interior H\"{o}lder estimates due to DiBenedetto and Friedman \cite{DF1985} to get H\"{o}lder continuity of $\tilde{\sigma}_k$-curvatures on the rescaled hypersurfaces.
This method was also used by Alessandroni and Sinestrari in \cite{AS2010} for the flow by powers of the scalar curvature and by Cabezas-Rivas and Sinestrari in \cite{CS2010} where the volume-preserving flow by powers of the $k$-th mean curvature was studied.
As mentioned in \S\ref{sec4.1}, we can use a technique of Tso \cite{Tso1985} to show that as long as the unrescaled hypersurfaces $M_t,~t\in[0,T']$ enclose a fixed ball $B_{r}(y_0)$ for some $y_0\in\mathbb{R}^{n+1}$ and $r>0$, the $\sigma_k$-curvature of $M_t$ has a positive upper bound depending on $r$.

\begin{lem}\label{upper0}
Let $X:M\times [0,T)$ be a smooth strictly convex solution of \eqref{flow} for $\epsilon=0$. If all the unrescaled hypersurfaces $M_t$ on a time interval $[0,T']~(T'<T)$ enclose a fixed ball $B_{r}(y_0)$ for some $y_0\in\mathbb{R}^{n+1}$ and $r>0$, then we have
\begin{equation*}
\sigma_k (p,t) \leq C(M_0,r,k,\alpha,n),~\textrm{for~all}~(p,t)\in M\times[0,T'].
\end{equation*}
\end{lem}
\begin{proof}
When $k=1$, this was proved by Schulze in \cite{SCHULZE2005,SCHULZE2006}. The case $k=2$ was prove by  Alessandroni and Sinestrari in \cite{AS2010}. For the case $k=n$, it was proved by Tso \cite{Tso1985} and Chow \cite{CHOW1985}.
For the general case, we use similar techniques.
We define $\psi=\frac{\sigma_k^\alpha}{\langle X-y_0,\nu\rangle-\frac{r}{2}}$, where $\langle \cdot,\cdot\rangle$ denotes the Euclidean metric. Since $M_t$ ($t\in [0,T']$)
enclose a fixed ball $B_{r}(y_0)$ for some $y_0\in\mathbb{R}^{n+1}$ and $r>0$,  we have $\langle X-y_0,\nu\rangle\geq r$, hence $\psi$ is
  well-defined on $[0,T']$. By using \eqref{flow}, \eqref{nu} and \eqref{2.5}, after a direct calculation, we have (cf. \cite{SCHULZE2005,CS2010})
\begin{equation*}
\begin{aligned}
\frac{\partial}{\partial t}\psi -\mathcal{L}\psi=&\frac{2}{\langle X-y_0,\nu\rangle-\frac{r}{2}} \dot{(\sigma_k^{\alpha})}^{ij}\nabla_i\psi\nabla_j\langle X-y_0,\nu\rangle\\
&+(k\alpha+1)\psi^2-\frac{r}{2}\psi^2\frac{\dot{(\sigma_k^{\alpha})}^{ml}h_{mr}h_l^r}{\sigma_k^{\alpha}}.
\end{aligned}
\end{equation*}
Since
$\sigma_k^{\alpha}=\psi(\langle X-y_0,\nu\rangle-\frac{r}{2})\geq \frac{r}{2}\psi$, using Lemma \ref{lem2.1} (iii), we then have
\begin{equation*}
\begin{aligned}
\frac{\partial}{\partial t}\psi -\mathcal{L}\psi\leq &\frac{2}{\langle X-y_0,\nu\rangle-\frac{r}{2}} \dot{(\sigma_k^{\alpha})}^{ij}\nabla_i\psi\nabla_j(\langle X-y_0,\nu\rangle)\\
&+(k\alpha+1-(r/2)^{1+\frac{1}{k\alpha}}\cdot\frac{k\alpha}{\binom{n}{k}^{1/k}}\psi^{\frac{1}{k\alpha}})\psi^2,
\end{aligned}
\end{equation*}
then by using the maximum principle, we obtain that
\begin{equation}\label{maxpsi}
\psi\leq \max\{\max_{M_0}\psi+1, (1+\frac{1}{k\alpha})^{k\alpha}\binom{n}{k}^\alpha(\frac{r}{2})^{-k\alpha-1}\},
\end{equation}
which gives an upper bound for $\sigma_k=(\psi(\langle X-y_0,\nu\rangle-\frac{r}{2}))^{1/\alpha}$:
\begin{equation*}
\sigma_k (p,t) \leq C(M_0,r,k,\alpha,n),~\textrm{for~all}~(p,t)\in M\times[0,T'].
\end{equation*}
\end{proof}

The following lemma gives the evolution of spheres in $\mathbb{R}^{n+1}$ along the $\sigma_k^{\alpha}$-flow \eqref{flow}.
\begin{lem}\label{lem4.1}
	Given $x_0\in \mathbb{R}^{n+1}, T\in \mathbb{R}^+$, we define
	\begin{equation*}
	\rho(t,T)=((k\alpha+1)\binom{n}{k}^\alpha (T-t))^{\frac{1}{k\alpha+1}},~\alpha>0,
	\end{equation*}
then the spheres $\partial B_{\rho(t,T)}(x_0)$ solve \eqref{flow} for $t\in[0,T)$, with $\rho_0=(T\cdot(k\alpha+1)\binom{n}{k}^\alpha)^{1/(k\alpha+1)}$ as the radius of the initial sphere.
\end{lem}
\begin{proof}
	Since the flow \eqref{flow} preserves the symmetry, in the sphere case, the equation \eqref{flow} reduces to the following ODE for the radius of the spheres:
	\begin{equation}\label{4.2}
	\left\{\begin{aligned}
	 \frac{d}{dt}\rho(t,T)=&~-\binom{n}{k}^{\alpha}\rho(t,T)^{-k\alpha},\\
	\rho(0,T)=&~\rho_0.
	\end{aligned}\right.
	\end{equation}
	Then the conclusion follows immediately by solving \eqref{4.2}.
\end{proof}

Lemma \ref{lem4.1} suggests us the following rescaling:
\begin{defn}
The rescaled immersions are defined by
\begin{equation}\label{4.3}
\tilde{X}(p,t):= (X(p,t)-q)\rho(t,T)^{-1}=((k\alpha+1)\binom{n}{k}^\alpha (T-t))^{-\frac{1}{k\alpha+1}}(X(p,t)- q),
\end{equation}
where $q$ is the point in $\mathbb{R}^{n+1}$ where the flow hypersurfaces contract to and $T$ is the maximal existence time of the flow.
\end{defn}
If there is no confusion, we will denote $\rho(t,T)$ by $\rho$ for short in the sequel. We use $\tilde{g}_{ij}$, $\tilde{A}=\{\tilde{h}_{ij}\}$ and $\tilde{\mathcal{W}}=\{\tilde{h}^i_j\}$ to denote  the components of  induced metric, the  second fundamental form and the Weingarten map of the rescaled hypersurfaces $\tilde{M}_t$, respectively. Then we have
\begin{equation}\label{relationg}
\begin{aligned}
&\tilde{g}_{ij}=\rho^{-2}g_{ij},~\tilde{g}^{ij}(p,t)=\rho^{2}g^{ij},\\
&\tilde{h}_{ij}=\rho^{-1}h_{ij},~\tilde{h}_i^j(p,t)=\rho h_i^j,
\end{aligned}
\end{equation}
and so on. We note that the Christoffel symbols $\tilde{\Gamma}_{ij}^k$ of the metric $\tilde{g}_{ij}$ of  the rescaled hypersurfaces are the same as that of the unrescaled hypersurfaces, so we still use
$\nabla$ to denote the  Levi-Civita connections on the rescaled hypersurfaces.

We define a new time function $\tau=\tau(t)$ by
\begin{equation}\label{tau}
\tau:= -(k\alpha+1)^{-1}\binom{n}{k}^{-\alpha} \log(1-\frac{t}{T}).
\end{equation}
Then  $\tau(0)=0$ and  $\tau$ ranges from $0$ to $\infty$. It is not difficult to obtain that the rescaled immersions satisfy the following evolution equation
\begin{equation*}
\frac{\partial \tilde{X}}{\partial \tau}=-\tilde{\sigma}_k^{\alpha}\nu+\binom{n}{k}^{\alpha}\tilde{X}.
\end{equation*}

By using \eqref{pinchingrho}, we can apply a similar argument to that in Lemma 7.2 of \cite{A1994CVPDE} to obtain uniform bounds for $\tilde{\rho}_-=\rho_-/\rho$ and $\tilde{\rho}_+=\rho_+/\rho$, the main idea is that by the comparison principle, the ball $B_{\rho(t,T)}(q)$ intersects the flow hypersurface $M_t$ for any $t\in[0,T)$.

\begin{lem}\label{lemrho}
\begin{equation*}
\frac{1}{C_4} \le \tilde{\rho}_- \le 1 \le \tilde{\rho}_+ \le C_4,~\textrm{for all}~ \tau \ge 0.
\end{equation*}
\end{lem}

Now, we apply Lemma \ref{upper0} to get a uniform upper bound for the $\tilde{\sigma}_k$-curvature of the rescaled hypersurface $\tilde{M}_t$.
For any time $T'<T$, we choose $r=\rho_-(T')$ in Lemma \ref{upper0} and let $y_0$ be the corresponding center of the inner ball with radius $\rho_-(T')$. When
 $T'\geq t_1 $, where $t_1$  is a fixed time which satisfies that $(1+\frac{1}{k\alpha})^{k\alpha}\binom{n}{k}^\alpha(\frac{\rho(t_1,T)}{2})^{-k\alpha-1}> \max_{M_0}\psi+1$,
 since $\rho_-(T')\leq \rho(T',T)\leq \rho(t_1,T)$, we obtain from \eqref{maxpsi} that
 \begin{equation*}
\psi(T')\leq (1+\frac{1}{k\alpha})^{k\alpha}\binom{n}{k}^\alpha(\frac{\rho_-(T')}{2})^{-k\alpha-1},
\end{equation*}
which implies that
 \begin{equation*}
 \begin{aligned}
\tilde{\sigma}_k(T')&=\sigma_k(T')\rho(T',T)^{k}=\Big(\psi(T')(\langle X(T')-y_0,\nu(T')\rangle-\frac{\rho_-(T')}{2})\Big)^{1/\alpha}\rho(T',T)^{k}\\
&\leq (2\rho_+(T') (1+\frac{1}{k\alpha})^{k\alpha}\binom{n}{k}^\alpha(\frac{\rho_-(T')}{2})^{-k\alpha-1})^{1/\alpha}\rho(T',T)^{k}\\
&=2^{k+2/\alpha}(1+\frac{1}{k\alpha})^{k}\binom{n}{k}\frac{\tilde{\rho}_+^{1/\alpha}}{(\tilde{\rho}_-)^{k+1/\alpha}}\leq C(C_4,k,\alpha,n).
\end{aligned}
\end{equation*}
When $T'\leq t_1$, since $[0,t_1]$ is a fixed finite time interval, we immediately get an upper bound $C$ for
$\tilde{\sigma}_k(T')$ on $[0,t_1]$, and $C$ only depends on  $M_0$ and $k,~\alpha$ and $n$.
Therefore, we obtain an upper bound for $\tilde{\sigma}_k(t)$ for all $t\in[0,T)$. Equivalently, we obtain that
there exists a constant $C$ which only depends on $M_0$, $k,~\alpha$ and $n$ such that
\begin{equation}\label{uppersigmak0}
\tilde{\sigma}_k(\tau)\leq C,~\forall~\tau\in [0,\infty).
\end{equation}
\eqref{uppersigmak0} in combination with the pinching estimate \eqref{3.14} gives a uniform bound for the principal  curvatures of the rescaled hypersurfaces:
\begin{equation}\label{upperlambdai}
\tilde{\lambda}_{\text{max}}(\tau)\leq C,~\forall~\tau\in[0,\infty),
\end{equation}
where $C$ differs from the one in \eqref{uppersigmak0} up to a universal constant.
In order to apply the interior H\"{o}lder estimates due to DiBenedetto and Friedman \cite{DF1985}, we first prove that we can rewrite the evolution of
 $\tilde{\sigma}_k$ in the following form.

\begin{lem}\label{lemkeyeq}
\begin{equation}\label{keyeq}
\begin{aligned}
\frac{\partial}{\partial \tau}\tilde{\sigma}_k &= D_i\Big(\frac{\alpha}{d}(\tilde{\sigma}_k)^{\frac{1-k}{k}}(\dot{\tilde{\sigma}}_k)^{ij}D_j\big((\tilde{\sigma}_k)^d)\Big)
+\Gamma^i_{il}(\dot{\tilde{\sigma}}_k)^{lj}\cdot\frac{\alpha}{d}(\tilde{\sigma}_k)^{\frac{1-k}{k}}D_j
(\tilde{\sigma}_k)^d\\
&+(\tilde{\sigma}_k)^\alpha(\dot{\tilde{\sigma}}_k)^{ml}\tilde{h}_{mr}\tilde{h}_l^r-\tilde{\sigma}_k\binom{n}{k}^\alpha,
\end{aligned}
\end{equation}
where $d=\alpha+\frac{k-1}{k}$ and $D_i$ are the derivatives with respect to the local coordinates.
\end{lem}
\begin{proof}
By using  \eqref{sigmakevo}, \eqref{4.2}, \eqref{relationg} and \eqref{tau}, we obtain the following evolution equation for $\tilde{\sigma}_k=\sigma_k\cdot\rho^k$.
\begin{equation}\label{sigmakevo3}
\begin{aligned}
\frac{\partial}{\partial \tau}\tilde{\sigma}_k =(\dot{\tilde{\sigma}}_k)^{ij}\nabla_i\nabla_j(\tilde{\sigma}_k)^{\alpha}+(\tilde{\sigma}_k)^{\alpha}(\dot{\tilde{\sigma}}_k)^{ml}\tilde{h}_{mr}\tilde{h}_l^r-k\binom{n}{k}^\alpha \tilde{\sigma}_k.
\end{aligned}
\end{equation}
We can rewrite the first term on the right-hand side of \eqref{sigmakevo3} in local coordinates as follows.
\begin{equation}\label{divfree}
\begin{aligned}
(\dot{\tilde{\sigma}}_k)^{ij}\nabla_i\nabla_j(\tilde{\sigma}_k)^{\alpha}=&
(\dot{\tilde{\sigma}}_k)^{ij}(D_iD_j(\tilde{\sigma}_k)^{\alpha}-\Gamma^l_{ij}D_l(\tilde{\sigma}_k)^{\alpha})\\
=&D_i\big((\dot{\tilde{\sigma}}_k)^{ij}D_j(\tilde{\sigma}_k)^{\alpha}\big)-D_i\big((\dot{\tilde{\sigma}}_k)^{ij}\big)D_j(\tilde{\sigma}_k)^{\alpha}
-(\dot{\tilde{\sigma}}_k)^{ij}\Gamma^l_{ij}D_l(\tilde{\sigma}_k)^{\alpha}\\
=&D_i\big((\dot{\tilde{\sigma}}_k)^{ij}D_j(\tilde{\sigma}_k)^{\alpha}\big)-\Big(\nabla_i\big((\dot{\tilde{\sigma}}_k)^{ij}\big)-\Gamma^i_{il} (\dot{\tilde{\sigma}}_k)^{lj}- \Gamma^j_{il} (\dot{\tilde{\sigma}}_k)^{il}\Big)D_j(\tilde{\sigma}_k)^{\alpha}\\
&-(\dot{\tilde{\sigma}}_k)^{ij}\Gamma^l_{ij}D_l(\tilde{\sigma}_k)^{\alpha}\\
=&D_i\big((\dot{\tilde{\sigma}}_k)^{ij}D_j(\tilde{\sigma}_k)^{\alpha}\big)+\Gamma^i_{il} (\dot{\tilde{\sigma}}_k)^{lj}D_j(\tilde{\sigma}_k)^{\alpha},
\end{aligned}
\end{equation}
where we used Lemma \ref{lem2.1} (iv) in the last equality.
Let $d=\alpha+\frac{k-1}{k}$, then we have
\begin{equation}\label{dj}
D_j(\tilde{\sigma}_k)^{\alpha}=\alpha (\tilde{\sigma}_k)^{\alpha-1}D_j(\tilde{\sigma}_k)=\frac{\alpha}{d}(\tilde{\sigma}_k)^{\frac{1-k}{k}}D_j(\tilde{\sigma}_k)^d.
\end{equation}
We obtain \eqref{keyeq} immediately by combining \eqref{sigmakevo3}, \eqref{divfree} and \eqref{dj}.
\end{proof}

\begin{rem}\label{approrem}
By using the pinching estimate \eqref{3.14} and the relation \eqref{relationg}, we know that the ratios of the principal curvatures of the rescaled hypersurfaces are also bounded from below and above by uniform positive constants, so we obtain that there exists a positive constant $C$ which only depends on $M_0$, $k,\alpha $ and $n$ such that
\begin{equation}\label{appro}
C^{-1}\tilde{g}^{ij}\leq(\tilde{\sigma}_k)^{\frac{1-k}{k}}(\dot{\tilde{\sigma}}_k)^{ij}\leq C \tilde{g}^{ij}.
\end{equation}
We denote a double bound like  \eqref{appro} by
$(\tilde{\sigma}_k)^{\frac{1-k}{k}}(\dot{\tilde{\sigma}}_k)^{ij}\approx \tilde{g}^{ij}.$
\end{rem}

\begin{lem}\label{lemint}
For the rescaled flow, there exists a constant $C$ which only depends on $M_0$, $k,~\alpha,~n$  such that for any $\tau_2>\tau_1>0$, we have
\begin{equation}\label{int}
\int_{\tau_1}^{\tau_2}\int_{\tilde{M}}|\nabla (\tilde{\sigma}_k^d)|^2d\tilde{\mu}d\tau\leq C(1+\tau_2-\tau_1),
\end{equation}
where $d=\alpha+\frac{k-1}{k}$.
\end{lem}
\begin{proof}
For $k=1$, this was proved by Schulze (see Lemma 3.3 in \cite{SCHULZE2006}). For general $k$, the proof is similar and the main idea is to use integration by parts.
Using \eqref{appro}, we have that
\begin{equation}\label{int1}
\begin{aligned}
\int_{\tilde{M}}|\nabla (\tilde{\sigma}_k^d)|^2d\tilde{\mu}&\leq C \int_{\tilde{M}}(\tilde{\sigma}_k)^{\frac{1-k}{k}}(\dot{\tilde{\sigma}}_k)^{ij}\nabla_i(\tilde{\sigma}_k^d)\nabla_j(\tilde{\sigma}_k^d)d\tilde{\mu}\\
&=C\int_{\tilde{M}}\frac{d}{\alpha}(\dot{\tilde{\sigma}}_k)^{ij}\nabla_i(\tilde{\sigma}_k^\alpha)\nabla_j(\tilde{\sigma}_k^d)d\tilde{\mu}\\
&=-C\cdot \frac{d}{\alpha}\int_{\tilde{M}}(\tilde{\sigma}_k^d)(\dot{\tilde{\sigma}}_k)^{ij}\nabla_i\nabla_j(\tilde{\sigma}_k^\alpha)d\tilde{\mu},
\end{aligned}
\end{equation}
where we used integration by parts and Lemma \ref{lem2.1} (iv)  in the last equality.
By using \eqref{sigmakevo3} and \eqref{int1}, we have
\begin{equation}\label{int2}
\begin{aligned}
\int_{\tilde{M}}|\nabla (\tilde{\sigma}_k^d)|^2d\tilde{\mu}&\leq
-C\cdot\frac{d}{\alpha(d+1)}\int_{\tilde{M}}\frac{\partial}{\partial \tau}(\tilde{\sigma}_k)^{d+1}d\tilde{\mu}\\
&+C\cdot\frac{d}{\alpha}\int_{\tilde{M}}(\tilde{\sigma}_k)^{d+\alpha}(\dot{\tilde{\sigma}}_k)^{ml}\tilde{h}_{mr}\tilde{h}_l^rd\tilde{\mu}
-C\cdot\frac{d}{\alpha}k\binom{n}{k}^\alpha\int_{\tilde{M}} (\tilde{\sigma}_k)^{d+1}d\tilde{\mu}.
\end{aligned}
\end{equation}
Since $\tilde{\sigma}_k$-curvature and the principal curvatures $\tilde{\lambda}_i$ of the rescaled hypersurfaces are uniformly bounded above for all $\tau>0$ (see \eqref{uppersigmak0},\eqref{upperlambdai}),
using \eqref{vol}, \eqref{4.2}, \eqref{tau} and \eqref{int2}, we obtain that there exist another constant $C'$  which does not depend on $\tau$ such that
\begin{equation*}
\begin{aligned}
\int_{\tilde{M}}|\nabla (\tilde{\sigma}_k^d)|^2d\tilde{\mu}\leq -C\cdot\frac{d}{\alpha(d+1)} \frac{d}{d\tau}\int_{\tilde{M}} (\tilde{\sigma}_k)^{d+1}d\tilde{\mu}+C',
\end{aligned}
\end{equation*}
from which we obtain the estimate \eqref{int} directly, since $\tilde{\sigma}_k$ is uniformly bounded from above and $\tilde{\mu}(\tilde{M})$ is bounded by $(n+1)\omega_{n+1}\tilde{\rho}_+^n$. Here $\omega_{n+1}$ is the volume of the unit $n$-ball.
\end{proof}

Armed with Lemma \ref{lemkeyeq} and Lemma \ref{lemint}, we can apply the interior H\"{o}lder estimates of  DiBenedetto and Friedman \cite{DF1985}, as proceeded by Schulze \cite{SCHULZE2006}, to obtain the following
H\"{o}lder estimate. The main idea is that the rescaled hypersurface can be locally written as a  graph with uniformly bounded $C^2$-norm. As the proof is similar to that of \cite[Lemma 3.4]{SCHULZE2006}, we omit the details here.
\begin{lem}\label{holder}
There exist universal constants $C>0,~\eta>0$ and $\beta\in(0,1)$ such that for every $(p,\tau)\in M\times[\eta,\infty)$, the $\beta$-H\"{o}lder norm in space-time of $\tilde{\sigma}_k$
on $B_{\eta}(p)\times(\tau-\eta,\tau+\eta)$ is bounded by $C$.
\end{lem}

Next, by replacing the estimate (2.3) in Theorem 2.6 of \cite{SCHULZE2006} by  our  pinching estimate \eqref{3.1}, following the same steps as in \cite{SCHULZE2006}
(cf. \cite{A1994CVPDE}),  by  using
the upper bound on $\tilde{\rho}_+$ in Lemma \ref{lemrho},  Lemma \ref{holder}, the H\"{o}lder estimate by Andrews \cite[Theorem 6]{andrews2004fully} (or the H\"{o}lder estimate in the case of one space dimension in \cite{lieberman1996second}), the parabolic Schauder estimates \cite{lieberman1996second} and interpolation inequalities, we  conclude that the rescaled flow  hypersurfaces converge in $C^{\infty}$-topology to the unit sphere $\mathbb{S}^n$.

Finally, in order to prove that the convergence is exponentially fast, we use analogous argument to that in Theorem 3.5 of \cite{SCHULZE2006}, the only difference is that we need to consider the evolution of the rescaled quantity $\tilde{G}$ instead of the evolution of $\tilde{f}$ in \cite{SCHULZE2006}. More precisely, since $G$ (see \eqref{defG}) is a homogeneous of degree $2(k\alpha-1)$ function of the principal curvatures, using the definition of the rescaling \eqref{4.3}, we have $\tilde{G}=G\cdot\rho(t,T)^{2(k\alpha-1)}$,  using \eqref{4.2}, we can calculate directly  and obtain that
\begin{equation*}
\begin{aligned}
\frac{\partial}{\partial \tau} \tilde{G}=\frac{\partial G}{\partial t} \cdot\rho(t,T)^{3k\alpha-1}- 2(k\alpha-1)\binom{n}{k}^{\alpha} \tilde{G},
\end{aligned}
\end{equation*}
then we get from Theorem \ref{thm3.1} that
\begin{equation}\label{tildeg}
\frac d{d\tau}\max_{\tilde{M}_{\tau}}\tilde{G} \leq - 2(k\alpha-1)\binom{n}{k}^{\alpha} \max_{\tilde{M}_{\tau}}\tilde{G}.
\end{equation}
By applying the maximum principle and using the fact that the rescaled hypersurfaces converge in $C^{\infty}$-topology to the unit sphere $\mathbb{S}^n$, \eqref{tildeg} implies that
there exists a positive constant $\delta_0$ such that
\begin{equation}\label{tildelambda}
|\tilde{\lambda}_{max}-\tilde{\lambda}_{min}|(p,\tau) \le C e^{-\delta_0 \tau} ,\quad\quad \forall ~(p,\tau) \in M \times [0,\infty).
\end{equation}
After obtaining \eqref{tildelambda}, we obtain the exponential convergence of the rescaled hypersurfaces by standard arguments as done in \cite{A1994CVPDE} and \cite{SCHULZE2006}.

\subsection{Convergence of the rescaled hypersurfaces for the flow \eqref{flow} with $\epsilon=1$}\label{sec4.3}
First, we show that as long as the unrescaled hypersurfaces $M_t,~t\in[0,T']$ enclose a fixed ball $B_{r}(y_0)$ for some $y_0\in\mathbb{S}^{n+1}$ and $r>0$, the $\sigma_k$-curvature of $M_t$ has a positive upper bound depending on $r$.  This is needed in  \S\ref{sec4.1}.
\begin{lem}\label{upper1}
 Let $X:M\times [0,T)$ a smooth strictly convex solution of \eqref{flow} for $\epsilon=1$.   If all the unrescaled hypersurfaces $M_t$ on a time interval $[0,T']~(T'<T)$ enclose a fixed ball $B_{r}(y_0)$ for some $y_0\in\mathbb{S}^{n+1}$ and $r<\pi/4$, then we have
\begin{equation*}
\sigma_k (p,t) \leq C(M_0,r,k,\alpha,n),~\textrm{for~all}~(p,t)\in M\times[0,T'].
\end{equation*}
\end{lem}
\begin{proof}
When $\alpha=1/k$, the conclusion is contained in the results  by Gerhardt \cite{Gerhardt2015}.
We consider the case $\alpha>1/k$. As all the unrescaled hypersurfaces $M_t$ ($t\in[0,T']$) enclose the ball $B_{r}(y_0)$, we can write $M_t~(t\in[0,T'])$ as
a graph in a geodesic polar coordinate system with center $y_0$:
\begin{equation*}
M_t=\text{graph}~ u(\cdot,t),~\forall~t\in[0,T'],
\end{equation*}
with $u(\cdot,t)\geq r$. As $u$ is decreasing with respect to $t$, we may assume that $u(\cdot,t)<\pi/2$ for $t\in[t_0,T']$.
Recall that the support function of  $M_t$ is defined by $\chi=\bar{g}(\sin{u}~\partial_r,\nu)$ (see \S\ref{sec2.2}).
Assume that $\chi$ attains a minimum at some point $p_t\in M_t$, as $M_t$ is strictly convex, $p_t$ is a critical point of $u$ (cf. \cite[Lemma 7.1]{Gerhardt2015}), we obtain that
$$\chi(p,t)\geq \chi(p_t,t)=\sin{u(p_t,t)}\geq \sin{r},~\forall~t\in[t_0,T'].$$
We define $\psi=\frac{\sigma_k^\alpha}{\chi-\sin{r}/2}$, then  $\psi$ is
  well-defined on $[t_0,T']$.
  By using \eqref{s2:evl-chi} and \eqref{2.5}, after a direct calculation, we have (cf. \cite{Gerhardt2015})
\begin{equation*}
\begin{aligned}
\frac{\partial}{\partial t}\psi -\mathcal{L}\psi=&\frac{2}{\chi-\sin{r}/2} \dot{(\sigma_k^{\alpha})}^{ij}\nabla_i\psi\nabla_j\chi\\
&+(k\alpha+1)\cos{u}\cdot \psi^2-\frac{\sin{r}}{2}\psi^2\frac{\dot{(\sigma_k^{\alpha})}^{ml}h_{mr}h_l^r}{\sigma_k^{\alpha}}+\epsilon \psi \dot{(\sigma_k^{\alpha})}^{ij}g_{ij}.
\end{aligned}
\end{equation*}
By Euler relation \eqref{2.2}, we have
$\sum_i \dot{(\sigma_k^{\alpha})}^{i}\lambda_i=k\alpha \sigma_k^{\alpha}$,
since we have uniform lower bounds for the principal curvatures,  then we obtain that
$\dot{(\sigma_k^{\alpha})}^{ij}g_{ij}=\sum_i \dot{(\sigma_k^{\alpha})}^{i}\leq C\sigma_k^{\alpha}$, where
$C$ is a positive constant which only depends on $M_0$, $k$, $\alpha$ and $n$.
We also have that
$\frac{\sin{r}}{2}\psi\leq \sigma_k^{\alpha}=\psi(\chi-\frac{\sin{r}}{2})\leq\psi $ on $[t_0,T']$, $\cos{u}\leq 1$, by using Lemma \ref{lem2.1} (iii), we obtain that
\begin{equation*}
\begin{aligned}
\frac{\partial}{\partial t}\psi -\mathcal{L}\psi\leq &\frac{2}{\chi-\sin{r}/2} \dot{(\sigma_k^{\alpha})}^{ij}\nabla_i\psi\nabla_j\chi\\
&+(b_1-b_2(\sin{r})^{1+\frac{1}{k\alpha}}\psi^{1/(k\alpha)})\psi^2,
\end{aligned}
\end{equation*}
where $b_1,~b_2$ are two constants which only depend on $M_0,~k,~\alpha$ and $n$.
Then by using the maximum principle, we obtain that
\begin{equation}\label{upppsi1}
\psi\leq \max\{\max_{M_0}\psi+1, b_3 (\sin{r})^{-(k\alpha+1)}\},
\end{equation}
where $b_3$ is another constant which only depends on $M_0,~k,~\alpha$ and $n$.
\eqref{upppsi1} gives an upper bound for $\sigma_k=(\psi(\chi-\sin{r}/2))^{1/\alpha}$ on $[t_0,T']$.
On the other hand, we also have an upper bound for $\sigma_k$ on the finite time interval $[0,t_0]$, so we obtain that
\begin{equation*}
\sigma_k \leq C(M_0,r,k,\alpha,n).
\end{equation*}
\end{proof}

Similar to the Euclidean case, if the initial hypersurface is a geodesic sphere, then the flow hypersurfaces of the flow \eqref{flow}
are all spheres with the same center and their radii $\Theta(t)$ satisfy the following  ODE
		\begin{equation}\label{Theta}
	 \frac{d}{dt}\Theta(t)=~-\binom{n}{k}^{\alpha}(\tan{\Theta(t)})^{-k\alpha}.
	\end{equation}
$\Theta(t)$ tends to $0$ in finite time, we denote by $\Theta(t,T)$ the radii of the sphere solution which shrinks to a point as $t\to T$, where $T$ is the maximal existence time of the flow \eqref{flow} with initial hypersurface $M_0$ for $\epsilon=1$. By using maximum principle, for any fixed time $t\in[0,T)$, the sphere with center $y_0\in\hat{M}_t$ (the convex body of $M_t$) and radius $\Theta(t,T)$ intersects $M_t$, so if we write $M_t$ as a graph in polar coordinates with center $y_0\in\hat{M}_t$, then we have the following relation among the graph function, the inner radius, the outer radius and $\Theta(t,T)$:
\begin{equation}\label{multirelation}
\inf_{M_t} u\leq \rho_-(t)\leq \Theta(t,T)\leq\rho_+(t)\leq \sup_{M_t} u\leq 2\rho_+(t).
\end{equation}
 If there is no confusion, we will denote $\Theta(t,T)$ by $\Theta$ for short in the sequel.
We note that when $t\in[t_\delta=T-\delta,T)$, for any $y_0\in\hat{M}_t$, we have
\begin{equation*}
\hat{M}_t\subset B_{2\rho_+(t)}(y_0) \subset B_{2C_4\rho_-(t)}(y_0),
\end{equation*}
 where $C_4$ is the constant in \eqref{pinchingrho}. We can choose $\delta$ small enough (without changing the notation) such that $$2C_4\rho_-(t_{\delta})\leq 2C_4\Theta(t_\delta,T)<1.$$

We define a new time parameter by
\begin{equation}\label{tau1}
\tau=-\log{\Theta(t,T)}.
\end{equation}
Then we have
\begin{equation}\label{dtau}
\frac{d\tau}{dt}=-\frac{1}{\Theta(t,T)}\frac{d}{dt}\Theta(t,T)=(\tan{\Theta})^{-k\alpha}\binom{n}{k}^\alpha\Theta^{-1}.
\end{equation}

Now, we apply Lemma \ref{upper1} to show that the rescaled $\sigma_k$-curvature $\tilde{\sigma}_k=\sigma_k\Theta^k$ and the rescaled principal curvatures
$\tilde{\lambda}_i=\lambda_i\Theta$ are bounded from above by uniform constants.
For any fixed $t_0\in(t_\delta,T)$,  let $B_{\rho_-(t_0)}(y_0)$ be an inner ball of $\hat{M}_{t_0}$. We write the flow hypersurface $M_t~(t\in[t_\delta,t_0])$ as
a graph in a geodesic polar coordinate system with center $y_0$:
\begin{equation*}
M_t=\text{graph}~ u(\cdot,t),~\forall~t\in[t_\delta,t_0].
\end{equation*}
Then the graph function satisfies that
\begin{equation*}
\rho_-(t_0)\leq u(t_0)\leq u(t)\leq 1.
\end{equation*}
Assume that $t_1>t_\delta$ is a fixed time which  satisfies that $b_3 (\sin{\Theta(t_1,T)})^{-(k\alpha+1)}> \max_{M_0}\psi+1$.
When
 $t_0\geq t_1$,
 since $\rho_-(t_0)\leq \Theta(t_0,T)\leq \Theta(t_1,T)$, from \eqref{upppsi1} we obtain that
 \begin{equation}
\psi(t_0)\leq b_3 (\sin{\rho_-(t_0)})^{-(k\alpha+1)},
\end{equation}
which implies that
 \begin{equation}
 \begin{aligned}
\sigma_k^\alpha(t_0)&=\psi(t_0)(\chi-\sin{\rho_-(t_0)}/2)\\
&\leq 2\rho_+(t_0)b_3 (\sin{\rho_-(t_0)})^{-(k\alpha+1)},
\end{aligned}
\end{equation}
note that $\sin{\rho_-(t_0)}\geq \frac{2}{\pi}\rho_-(t_0)$,
then by using \eqref{pinchingrho} and \eqref{multirelation}, we obtain that
 \begin{equation}
 \begin{aligned}
\tilde{\sigma}_k(t_0)&=\sigma_k(t_0)\Theta(t_0,T)^{k}\leq C(C_4,k,\alpha,n).
\end{aligned}
\end{equation}
When $t_0\leq t_1$, since $[t_\delta,t_1]$ is a fixed finite time interval, we immediately get an upper bound $C'$ for
$\tilde{\sigma}_k(t_0)$ on $[t_\delta,t_1]$, and $C'$ only depends on  $M_0$ and $k,~\alpha$ and $n$.
Therefore, we obtain an upper bound for $\tilde{\sigma}_k(t)$ for all $t\in[t_\delta,T)$. Equivalently, we obtain that
there exists a constant $C$ which only depends on $M_0$, $k,~\alpha$ and $n$ such that
\begin{equation}\label{uppersigmak1}
\tilde{\sigma}_k(\tau)\leq C,~\forall~\tau\in [\tau_\delta,\infty),~\text{with}~\tau_\delta=-\log{\Theta(t_\delta,T)}.
\end{equation}
\eqref{uppersigmak1} in combination with the pinching estimate \eqref{3.14} gives a uniform bound for the  rescaled principal  curvatures:
\begin{equation}\label{upperlambda}
\tilde{\lambda}_{\text{max}}(\tau)\leq C,~\forall~\tau\in[\tau_\delta,\infty),
\end{equation}
where $C$ differs from the one in \eqref{uppersigmak1} up to a universal constant.

Let $t_1\in[t_\delta,T)$ be arbitrary and let $t_2>t_1$ be the time which satisfies
\begin{equation*}
\Theta(t_2,T)=\frac{1}{2}\Theta(t_1,T),
\end{equation*}
then $\tau_i=-\log{\Theta(t_i,T)}~(i=1,2)$ satisfy $\tau_2=\tau_1+\log{2}$.
Let $y_0$ be the center of an inner ball of $\hat{M}_{t_2}$, we introduce polar coordinates with center $y_0$ and write $M_t$ as graph of $u(\theta,t)$ for $t\in[t_1,t_2]$,
by using a similar argument to that in \cite[\S 7]{Gerhardt2015},
we have the following Lemma.
\begin{lem}[cf. \cite{Gerhardt2015}]\label{lem4.10}
\begin{itemize}
\item[(i)]There exists a constant $c$ which only depends on $M_0,~k,~\alpha$ and $n$ such that for all $(\theta,t)\in\mathbb{S}^n\times[t_1,t_2]$, we have
\begin{equation*}
c^{-1}\Theta(t_2,T)\leq u(\theta,t)\leq c~\Theta(t_2,T),~u_{max}(t)\leq c^2u_{min}(t).
\end{equation*}
\item[(ii)] $ v^2=1+\sin{u}^{-2}|Du|^2_{g_{\mathbb{S}^n}}$ is uniformly bounded in $\mathbb{S}^n\times [t_1,t_2]$.
\item[(iii)] $\Gamma_{ij}^k-\bar{\Gamma}_{ij}^k$ is uniformly bounded, where $\Gamma_{ij}^k$ (resp. $\bar{\Gamma}_{ij}^k$) are the Christoffel symbols of  the metric $g_{ij}$ of $M_t$ (resp. the standard sphere metric $\sigma_{ij}$).
\end{itemize}
\end{lem}
The following Lemma tells us that we can write the evolution equation of $\tilde{\sigma}_k$ in a special form similar to Euclidean case.
\begin{lem}\label{lemkeyeq1}
\begin{equation}\label{keyeq1}
\begin{aligned}
\frac{\partial}{\partial \tau}\tilde{\sigma}_k &= \binom{n}{k}^{-\alpha}\Theta^{-k\alpha}(\tan{\Theta})^{k\alpha}\cdot \bar{\nabla} _i\big(\frac{\alpha}{d}\Theta^2(\sigma_k)^{\frac{1-k}{k}}\dot{\sigma_k}^{ij} \bar{\nabla} _j\big((\tilde{\sigma}_k)^d)\big)\\
&+\binom{n}{k}^{-\alpha}\Theta^{-k\alpha}(\tan{\Theta})^{k\alpha}\cdot(\Gamma^i_{il}-\tilde{\Gamma}^i_{il})\Theta^{k+1}\dot{\sigma_k}^{lj} \cdot\frac{\alpha}{d}(\tilde{\sigma}_k)^{\frac{1-k}{k}}\bar{\nabla} _j
(\tilde{\sigma}_k)^d\\
&+\binom{n}{k}^{-\alpha}\Theta^{-k\alpha}(\tan{\Theta})^{k\alpha}(\tilde{\sigma}_k)^\alpha\Big(
\dot{\sigma_k}^{ml}h_{mr}h_l^r\Theta^{k+1}+\epsilon\dot{\sigma_k}^{ij}h_{ij}\Theta^{k+1}\Big)
-k \tilde{\sigma}_k,
\end{aligned}
\end{equation}
where $d=\alpha+\frac{k-1}{k}$ and $\bar{\nabla}$ denotes the Levi-Civita connection on the unit sphere $\mathbb{S}^n$.
\end{lem}
\begin{proof}
By using \eqref{sigmakevo},  \eqref{Theta} and \eqref{dtau}, we obtain  the following evolution equation for $\tilde{\sigma}_k=\sigma_k\Theta^k$.
\begin{equation}\label{sigmakevo3n}
\begin{aligned}
\frac{\partial}{\partial \tau}\tilde{\sigma}_k &=\binom{n}{k}^{-\alpha}\Theta^{-k\alpha}(\tan{\Theta})^{k\alpha}\Theta^{k+1}\dot{\sigma_k}^{ij}\nabla_i\nabla_j(\tilde{\sigma}_k^{\alpha})\\
&+\binom{n}{k}^{-\alpha}\Theta^{-k\alpha}(\tan{\Theta})^{k\alpha}(\tilde{\sigma}_k)^\alpha\Big(
\dot{\sigma_k}^{ml}h_{mr}h_l^r\Theta^{k+1}+\epsilon\dot{\sigma_k}^{ij}h_{ij}\Theta^{k+1}\Big)
-k \tilde{\sigma}_k.
\end{aligned}
\end{equation}
We can rewrite the first term on the right-hand side of \eqref{sigmakevo3n} in local coordinates as follows (here $D_i$ are the ordinary derivatives with respect to the local coordinates).
\begin{equation}\label{divfreen}
\begin{aligned}
\dot{\sigma_k}^{ij}\nabla_i\nabla_j(\tilde{\sigma}_k)^{\alpha}=&
\dot{\sigma_k}^{ij}(D_iD_j(\tilde{\sigma}_k)^{\alpha}-\Gamma^l_{ij}D_l(\tilde{\sigma}_k)^{\alpha})\\
=&D_i\big(\dot{\sigma_k}^{ij}D_j(\tilde{\sigma}_k)^{\alpha}\big)-D_i\big(\dot{\sigma_k}^{ij}\big)D_j(\tilde{\sigma}_k)^{\alpha}
-\dot{\sigma_k}^{ij}\Gamma^l_{ij}D_l(\tilde{\sigma}_k)^{\alpha}\\
=&D_i\big(\dot{\sigma_k}^{ij}D_j(\tilde{\sigma}_k)^{\alpha}\big)-\Big(\nabla_i\big(\dot{\sigma_k}^{ij}\big)-\Gamma^i_{il} \dot{\sigma_k}^{lj}- \Gamma^j_{il} \dot{\sigma_k}^{il}\Big)D_j(\tilde{\sigma}_k)^{\alpha}\\
&-\dot{\sigma_k}^{ij}\Gamma^l_{ij}D_l(\tilde{\sigma}_k)^{\alpha}\\
=&D_i\big(\dot{\sigma_k}^{ij}D_j(\tilde{\sigma}_k)^{\alpha}\big)+\Gamma^i_{il} \dot{\sigma_k}^{lj}D_j(\tilde{\sigma}_k)^{\alpha}\\
=&\bar{\nabla}_i\big(\dot{\sigma_k}^{ij}\bar{\nabla}_j(\tilde{\sigma}_k)^{\alpha}\big)+(\Gamma^i_{il}-\bar{\Gamma}^i_{il}) \dot{\sigma_k}^{lj}\bar{\nabla}_j(\tilde{\sigma}_k)^{\alpha},
\end{aligned}
\end{equation}
where we used Lemma \ref{lem2.1} (iv) in the fourth equality.
Let $d=\alpha+\frac{k-1}{k}$, then we have
\begin{equation}\label{djn}
D_j(\tilde{\sigma}_k)^{\alpha}=\alpha (\tilde{\sigma}_k)^{\alpha-1}D_j(\tilde{\sigma}_k)=\frac{\alpha}{d}(\tilde{\sigma}_k)^{\frac{1-k}{k}}D_j\big((\tilde{\sigma}_k)^d).
\end{equation}
We obtain \eqref{keyeq1} immediately by combining \eqref{sigmakevo3n}, \eqref{divfreen} and \eqref{djn}.
\end{proof}
\begin{rem}\label{rem4.12}
Note that $\Theta$ is  small on the time interval $[t_\delta,T)$, so $\Theta$ is comparable with $\tan{\Theta}$ on $[t_\delta,T)$, we have a uniform bound
on $\Theta^{-k\alpha}(\tan{\Theta})^{k\alpha}$. Due to the pinching estimate \eqref{3.14}, we have
$\sigma_k^{\frac{1-k}{k}}\dot{\sigma_k}^{ij}\approx g^{ij}$. By using \eqref{pinchingrho}, \eqref{multirelation} and Lemma \ref{lem4.10}, we obtain that $\Theta^2g^{ij}\approx \Theta^2\sin{u}^{-2}\sigma^{ij}$, which is uniformly bounded. Hence, we have that $\Theta^2(\sigma_k)^{\frac{1-k}{k}}\dot{\sigma_k}^{ij}\approx \sigma^{ij}$. Here $\approx$ has the meaning as explained in Remark \ref{approrem}.
\end{rem}

\begin{lem}\label{lemint1}
There exists a constant $C$ which only depends on $M_0$, $k,~\alpha$ and $n$  such that for any $\tau_1\geq\tau_{\delta}$ and $\tau_2=\tau_1+\log{2}$, we have
\begin{equation*}
\int_{\tau_1}^{\tau_2}\int_{\mathbb{S}^n}|\bar{\nabla} (\tilde{\sigma}_k^d)|^2d\mu_{\mathbb{S}^n}d\tau\leq C(1+\tau_2-\tau_1),
\end{equation*}
where $d=\alpha+\frac{k-1}{k}$ and $\bar{\nabla}$ denotes the Levi-Civita connection on the unit sphere $\mathbb{S}^n$.
\end{lem}
\begin{proof}
Since $(\sigma_k)^{\frac{1-k}{k}}\dot{\sigma_k}^{ij}\approx g^{ij}$, we have that
\begin{equation}\label{int1n}
\begin{aligned}
\int_{M_t}|\nabla (\tilde{\sigma}_k^d)|^2d\mu_t&\leq C \int_{M_t}(\sigma_k)^{\frac{1-k}{k}}\dot{\sigma_k}^{ij}\nabla_i(\tilde{\sigma}_k^d)\nabla_j(\tilde{\sigma}_k^d)d\mu_t\\
&=C\cdot \int_{M_t}\frac{d}{\alpha}\Theta^{k-1}\dot{\sigma_k}^{ij}\nabla_i(\tilde{\sigma}_k^\alpha)\nabla_j(\tilde{\sigma}_k^d)d\mu_t\\
&=-C\cdot\frac{d}{\alpha}\int_{M_t}\Theta^{k-1}(\tilde{\sigma}_k^d)\dot{\sigma_k}^{ij}\nabla_i\nabla_j(\tilde{\sigma}_k^\alpha)d\mu_t,
\end{aligned}
\end{equation}
where we used integration by parts and Lemma \ref{lem2.1} (iv) in the last equality, and $C$ is a constant depending on $M_0,~k,~\alpha$ and $n$.
Using \eqref{vol}, \eqref{Theta} and \eqref{dtau}, we have
\begin{equation*}
\begin{aligned}
\frac{d}{d\tau}\int_{M_t}\Theta^{-n}\tilde{\sigma}_k^{d+1}d\mu_t=&-\int_{M_t}\binom{n}{k}^{-\alpha} \Theta^{1-n}(\tan{\Theta})^{k\alpha}\tilde{\sigma}_k^{d+1}\sigma_k^\alpha Hd\mu_t\\
&+n \int_{M_t}  \Theta^{-n}\tilde{\sigma}_k^{d+1}d\mu_t+\int_{M_t} \Theta^{-n}\frac{\partial}{\partial \tau}(\tilde{\sigma}_k)^{d+1}d\mu_t.
\end{aligned}
\end{equation*}
Then by using \eqref{sigmakevo3n}, we have
\begin{equation}\label{sigmakevo5}
\begin{aligned}
\frac{d}{d\tau}\int_{M_t}\Theta^{-n}\tilde{\sigma}_k^{d+1}d\mu_t=&-\int_{M_t}\binom{n}{k}^{-\alpha} \Theta^{1-n}(\tan{\Theta})^{k\alpha}\tilde{\sigma}_k^{d+1}\sigma_k^\alpha Hd\mu_t+n \int_{M_t}  \Theta^{-n}\tilde{\sigma}_k^{d+1}d\mu_t
\\
&+\int_{M_t}\Theta^{-n}(d+1)(\tilde{\sigma}_k)^{d}
\binom{n}{k}^{-\alpha}\Theta^{-k\alpha}(\tan{\Theta})^{k\alpha}\Theta^{k+1}\dot{\sigma_k}^{ij}\nabla_i\nabla_j(\tilde{\sigma}_k^{\alpha})d\mu_t\\
&+\int_{M_t}\Theta^{-n}(d+1)(\tilde{\sigma}_k)^{d}\binom{n}{k}^{-\alpha}\Theta^{-k\alpha}(\tan{\Theta})^{k\alpha}(\tilde{\sigma}_k)^\alpha
\dot{\sigma_k}^{ml}h_{mr}h_l^r\Theta^{k+1}d\mu_t\\
&+\int_{M_t}\Theta^{-n}(d+1)(\tilde{\sigma}_k)^{d}\binom{n}{k}^{-\alpha}
\Theta^{-k\alpha}(\tan{\Theta})^{k\alpha}(\tilde{\sigma}_k)^\alpha\epsilon\dot{\sigma_k}^{ij}h_{ij}\Theta^{k+1}d\mu_t\\
&-\int_{M_t}\Theta^{-n}(d+1)(\tilde{\sigma}_k)^{d}\cdot k \tilde{\sigma}_kd\mu_t.
\end{aligned}
\end{equation}
From \eqref{int1n} and \eqref{sigmakevo5}, we have
\begin{equation}\label{int2n}
\begin{aligned}
&\int_{M_t}\frac{\alpha}{d}(d+1)\Theta^{-k\alpha}(\tan{\Theta})^{k\alpha}\binom{n}{k}^{-\alpha}\Theta^{2-n}|\nabla (\tilde{\sigma}_k^d)|^2d\mu_t\\
\leq & -C\cdot\int_{M_t}\Theta^{-n}(d+1)(\tilde{\sigma}_k)^{d}
\binom{n}{k}^{-\alpha}\Theta^{-k\alpha}(\tan{\Theta})^{k\alpha}\Theta^{k+1}\dot{\sigma_k}^{ij}\nabla_i\nabla_j(\tilde{\sigma}_k^{\alpha})d\mu_t\\
=&C\cdot\Big(-\frac{d}{d\tau}\int_{M_t}\Theta^{-n}\tilde{\sigma}_k^{d+1}d\mu_t-\int_{M_t}\binom{n}{k}^{-\alpha} \Theta^{1-n}(\tan{\Theta})^{k\alpha}\tilde{\sigma}_k^{d+1}\sigma_k^\alpha Hd\mu_t+n \int_{M_t}  \Theta^{-n}\tilde{\sigma}_k^{d+1}d\mu_t\\
&+\int_{M_t}\Theta^{-n}(d+1)(\tilde{\sigma}_k)^{d}\binom{n}{k}^{-\alpha}\Theta^{-k\alpha}(\tan{\Theta})^{k\alpha}(\tilde{\sigma}_k)^\alpha
\dot{\sigma_k}^{ml}h_{mr}h_l^r\Theta^{k+1}d\mu_t\\
&+\int_{M_t}\Theta^{-n}(d+1)(\tilde{\sigma}_k)^{d}\binom{n}{k}^{-\alpha}
\Theta^{-k\alpha}(\tan{\Theta})^{k\alpha}(\tilde{\sigma}_k)^\alpha\epsilon\dot{\sigma_k}^{ij}h_{ij}\Theta^{k+1}d\mu_t\\
&-\int_{M_t}\Theta^{-n}(d+1)(\tilde{\sigma}_k)^{d}\cdot k \tilde{\sigma}_kd\mu_t\Big),
\end{aligned}
\end{equation}
Since $\tilde{\sigma}_k$-curvature and the principal curvatures $\tilde{\lambda}_i$ of the rescaled hypersurfaces are uniformly bounded from above for all $\tau>0$,
the volume of $M_t$ is comparable with $\Theta^n$,
from \eqref{int2n}, we obtain that there exist another constants $C'$ which only depends on $M_0,~k,~\alpha$ and $n$ such that
\begin{equation*}
\begin{aligned}
&\int_{M_t}\frac{\alpha}{d}(d+1)\Theta^{-k\alpha}(\tan{\Theta})^{k\alpha}\binom{n}{k}^{-\alpha}\Theta^{2-n}|\nabla (\tilde{\sigma}_k^d)|^2d\mu_t\\
\leq& -C\frac{d}{d\tau}\int_{M_t}\Theta^{-n}\tilde{\sigma}_k^{d+1}d\mu_t+C'.
\end{aligned}
\end{equation*}
Consequently,  we get
\begin{equation*}
\begin{aligned}
&\int_{\tau_1}^{\tau_2}\int_{M_t}\frac{\alpha}{d}(d+1)\Theta^{-k\alpha}(\tan{\Theta})^{k\alpha}\binom{n}{k}^{-\alpha}\Theta^{2-n}|\nabla (\tilde{\sigma}_k^d)|^2d\mu_t\\
\leq& C(1+\tau_2-\tau_1),
\end{aligned}
\end{equation*}
since $\tilde{\sigma}_k$ is uniformly bounded from above and $\int_{M_t}\Theta^{-n}d\mu_t$ is comparable with $\int_{\mathbb{S}^n}d\mu_{\mathbb{S}^n}$, here $C$ is another constant
which only depends on $M_0,~k,~\alpha$ and $n$.
Using Remark \ref{rem4.12}, we obtain that
$\int_{\tau_1}^{\tau_2}\int_{\mathbb{S}^n}|\bar{\nabla} (\tilde{\sigma}_k^d)|^2d\mu_{\mathbb{S}^n}d\tau$ is comparable with
$$\int_{\tau_1}^{\tau_2}\int_{M_t}\frac{\alpha}{d}(d+1)\Theta^{-k\alpha}(\tan{\Theta})^{k\alpha}\binom{n}{k}^{-\alpha}\Theta^{2-n}|\nabla (\tilde{\sigma}_k^d)|^2d\mu_t.$$
Then Lemma \ref{lemint1} follows immediately.
\end{proof}
Armed with Lemma \ref{lemkeyeq1}, Remark \ref{rem4.12} and Lemma \ref{lemint1}, we conclude by using the interior H\"{o}lder estimates of  DiBenedetto and Friedman \cite{DF1985}
that
\begin{lem}\label{holder1}
There exist universal constants $C>0,~\eta>0$ and $\beta\in(0,1)$ such that for every $(p,\tau)\in M\times[\tau_\delta+\eta,\infty)$, the $\beta$-H\"{o}lder norm in space-time of $\tilde{\sigma}_k$
on $B_{\eta}(p)\times(\tau-\eta,\tau+\eta)$ is bounded by $C$.
\end{lem}

In the remaining part of the this section, we finish the proof of Theorem \ref{thm1.2}.
Let $q\in\mathbb{S}^{n+1}$ be the point that the flow hypersurfaces $M_t$ shrink to as $t$ approaches $T$, we introduce geodesic polar coordinates with center $q$.
We will prove that the rescaled function $\tilde{u}(p,\tau)=u(p,t)\Theta(t,T)^{-1}$  converges exponentially in  $C^{\infty}(\mathbb{S}^{n})$ to the constant function $1$ as $\tau\to \infty$.
First, we note that although the rescaled principal curvatures $\tilde{\lambda}_i(p,t)=\lambda_i(p,t) \cdot\Theta(t,T)$ are not the principal curvatures of the graph of $\tilde{u}(p,t)$, they are closely related. From the expression of $h^i_j$ (see \eqref{hij}), the uniform upper bound on $\tilde{\lambda}_{\text{max}}$ (see \eqref{upperlambda}) and the $C^0,~C^1$ estimates of $\tilde{u}$ in Lemma \ref{lem4.10}, we obtain uniform
$C^2$-estimate of $\tilde{u}$ with respect to the metric of the unit sphere $\mathbb{S}^n$.
At each fixed  time $\tau_j$, we take a point $p_j\in\mathbb{S}^{n}$ such that $\tilde{u}(p_j,\tau_j)$ attains a maximum at $p_j$, then we have
$\tilde{\lambda}_i(p_j,\tau_j)=\lambda_i(p_j,\tau_j)\cdot\Theta(t_j,T)\geq \tan{u(p_j,\tau_j)}^{-1}\Theta(t_j,T)\geq C_5>0$, where $C_5$ is a constant which only depends on $M_0$, and $\tau_j$ and $t_j$ are related by $\tau_j=-\log{\Theta(t_j,T)}$.
For each $(p_j,\tau_j)$,  we have that
$\tilde{\sigma}_k(p_j,\tau_j)\geq \tilde{C}=\binom{n}{k}C_5^k$, then the H\"{o}lder  estimate (Lemma \ref{holder1}) implies that there exists a constant $\eta>0$ which does not depend on the time sequence $\tau_j$
 such that $\tilde{\sigma}_k\geq \tilde{C}/2$ on $B_{\eta}(p_j)\times[\tau_j-\eta,\tau_j+\eta]$. This together with the pinching estimate \eqref{3.14} and the uniform upper bound on $\tilde{\lambda}_i$ (see \eqref{upperlambda}) implies that
 $\tilde{u}$ satisfies a uniform parabolic equation on $B_{\eta}(p_j)\times[\tau_j-\eta,\tau+\eta]$ by using \eqref{s2:ut-evl}, \eqref{Theta} and \eqref{dtau}:
 \begin{equation}\label{tildeu}
 \frac{\partial}{\partial\tau}\tilde{u}=-\binom{n}{k}^{-\alpha}\Theta^{-k\alpha}(\tan{\Theta})^{k\alpha}v\tilde{\sigma}_k^\alpha+\tilde{u},
 \end{equation}
 where $v$ is the function defined in \eqref{s2:v-def}.
Since the speed function can be written in the form $F=\sigma_k^{\alpha}=(\sigma_k^{1/k})^{k\alpha}$ and $\sigma_k^{1/k}$ is a concave function of the principal curvatures,
we can apply the  H\"{o}lder estimate by Andrews \cite[Theorem 6]{andrews2004fully} (we can also apply the H\"{o}lder estimate in the case of one space dimension in  \cite{lieberman1996second} since the graph function only depends on one space variable) and
parabolic Schauder estimate \cite{lieberman1996second} to get uniform $C^\infty$ estimates of $\tilde{u}$ on $B_{\eta/2}(p_j)\times[\tau_j-\eta/2,\tau_j+\eta/2]$.
Since $\mathbb{S}^{n}$ is compact, there exists a subsequence (again denoted by $\tau_j$) such that $\{p_j\}$ converge to a point $p_\infty\in\mathbb{S}^{n}$, then we obtain uniform $C^\infty$ estimates of $\tilde{u}$ on $B_{\eta/2}(p_\infty)\times[\tau_j-\eta/2,\tau_j+\eta/2]$.
On the other hand, using our pinching estimate \eqref{3.1}, we have
\begin{equation*}
0 \le \frac{\lambda_{max}}{\lambda_{min}}+\frac{\lambda_{min}}{\lambda_{max}}-2 \le \frac{C}{\sigma_k^{2(\alpha-\frac{1}{k})}}= \frac{C\Theta^{2k(\alpha-\frac{1}{k})}}{\tilde{\sigma}_k^{2(\alpha-\frac{1}{k})}},
\end{equation*}
which together with the uniform positive lower bound on $\tilde{\sigma}_k$ implies that
\begin{equation*}
1 \le \frac{\lambda_{max}(p,\tau_j)}{\lambda_{min}(p,\tau_j)}\leq 1+C \Theta^{k(\alpha-\frac{1}{k})}=1+Ce^{-k(\alpha-\frac{1}{k})\tau_j},~\forall ~p\in B_{\eta/2}(p_\infty).
\end{equation*}
This in combination with the uniform upper bound on $\tilde{\lambda}_{\text{max}}$ (see \eqref{upperlambda}) implies that the trace-less part of the rescaled second fundamental form of $M_{t_j}$ has the following exponential decay
\begin{equation}\label{tracelessA}
  |\mathring{\tilde{A}}|(p,\tau_j)\leq Ce^{-k(\alpha-\frac{1}{k})\tau_j},\quad \forall~p\in B_{\eta/2}(p_\infty),
\end{equation}
where $C=C(M_0)$.
The $C^\infty$-estimates of $\tilde{u}$ and the fact that $\tilde{g}^{ij}=\Theta^2 g^{ij}\approx \sigma^{ij}$ imply uniform estimates for
\begin{equation*}
  |\nabla^m\tilde{A}|^2:=\tilde{g}^{i_1j_1}\cdots \tilde{g}^{i_mj_m}\nabla_{i_1i_2\cdots i_m}(\Theta h_k^l)~\nabla_{j_1j_2\cdots j_m}(\Theta h_l^k)
\end{equation*}
for all $m\in \mathbb{N}$, where $\nabla$ is the Levi-Civita connection with respect to the metric $g$ on $M_{t_j}$.
This together with \eqref{tracelessA} implies by interpolation that
\begin{equation*}
  |\nabla\mathring{\tilde{A}}|^2(p,\tau_j)\leq Ce^{-k(\alpha-\frac{1}{k})\tau_j},\quad \forall~p\in B_{\eta/2}(p_\infty).
\end{equation*}
By using the inequality  (cf. \cite[\S 2]{HUISKEN1984})
$
  |\nabla H|^2\leq \frac {n+2}{3}|\nabla A|^2
$,
we have
\begin{equation*}
  |\nabla A|^2=|\nabla\mathring{A}|^2+\frac {1}{n}|\nabla H|^2\leq |\nabla\mathring{A}|^2+\frac {n+2}{3n}|\nabla A|^2.
\end{equation*}
Therefore, we obtain
\begin{equation*}
\begin{aligned}
  |\nabla \tilde{A}|^2(p,\tau_j)&=\Theta^4|\nabla A|^2(p,\tau_j)\leq ~\Theta^4\cdot\frac{3n}{2n-2}|\nabla\mathring{A}|^2(p,\tau_j)\\
  &=\frac{3n}{2n-2}|\nabla\mathring{\tilde{A}}|^2(p,\tau_j)\leq ~Ce^{-k(\alpha-\frac{1}{k})\tau_j},\quad \forall~p\in B_{\eta/2}(p_\infty).
  \end{aligned}
\end{equation*}
This leads to the following estimate
\begin{equation*}
\begin{aligned}
&(\max_{B_{\eta/2}(p_\infty)}\tilde{\sigma}_k(p,\tau_j)-\min_{B_{\eta/2}(p_\infty)}\tilde{\sigma}_k(p,\tau_j))\\
\leq & C \Theta^k|A|^{k-1}|\nabla A|\text{diam}(M_{t_j}\cap \text{graph}~u(\cdot,t_j)|_{B_{\eta/2}(p_\infty)})\\
= &  C |\tilde{A}|^{k-1}|\nabla \tilde{A}|\Theta^{-1}\text{diam}(M_{t_j}\cap \text{graph}~u(\cdot,t_j)|_{B_{\eta/2}(p_\infty)})\\
\leq & Ce^{-\frac{k}{2}(\alpha-\frac{1}{k})\tau_j},
\end{aligned}
\end{equation*}
where $\tau_j=-\log{\Theta(t_j,T)}$.
Therefore, $\tilde{\sigma}_k(p,\tau_j)$ becomes arbitrary close to  $\tilde{\sigma}_k(p_\infty,\tau_j)$ in $B_{\eta/2}(p_\infty)$ as $\tau_j\to\infty$. Using the uniform H\"{o}lder estimate of $\tilde{\sigma}_k$, we obtain that for $j$ large, we can obtain that $\tilde{\sigma}_k\geq \tilde{C}/2$ holds on a larger region. If we repeat the same argument as above,
then we can extend the region where $\tilde{u}(p,\tau_j)$ has uniform $C^{\infty}$ estimates. As $\mathbb{S}^n$ is compact,
after finite steps, we obtain that $\tilde{u}(p,\tau_j)$ has uniform $C^{\infty}$ estimates on $\mathbb{S}^n$.
The above argument applies to any sequence $\tau_j$,
and the estimates do not depend on the sequence $\tau_j$, hence we obtain that $\tilde{u}(\cdot,\tau)$  obeys uniform a priori estimates in $C^{\infty}(\mathbb{S}^n)$ independently of $\tau$. Then by using the pinching estimates and the interpolation inequality, we get that $|\nabla\tilde{A}|^2(p,\tau)\leq Ce^{-k(\alpha-\frac{1}{k})\tau}$
for $\tau\in[\tau_\delta,\infty)$. After a similar argument to that in Section 8 of \cite{Gerhardt2015} (see Lemma 8.12, Corollary 8.13 and Theorem 8.14), we obtain that $\tilde{u}(\cdot,\tau)$ converges exponentially fast to the constant function $1$ in $C^\infty$-topology as $\tau\to \infty$.
This completes the proof of Theorem \ref{thm1.2}.

\appendix
\section{Sturm's Theorem and the Computer Algorithm}\label{app}
In this section, we will prove that for each $n$ and any fixed $k$ with $1\leq k\leq n$, there exists a constant $ c_0(n,k) >\frac{1}{k}$ such that
the polynomial $Q$ defined by \eqref{Q} is non-positive for any $x>0$ and $\alpha= c_0(n,k) $. This is needed in the proof of Theorem \ref{thm3.1}.
Note that in order to make sure that $Q$ defined by \eqref{Q} is non-positive for any $x>0$ and $\alpha= c_0(n,k) $, the highest coefficient of $Q$ needs to be non-positive, which implies that
\begin{equation}\label{estalpha}
\frac{1}{k} \le  c_0(n,k)  \le \frac{1}{k-1},~\forall~k\geq 2.
\end{equation}

In the case $1<k\le n \le k^2$, we  prove that $ c_0(n,k) =\frac{1}{k-1}$ satisfies that $Q$ is non-positive for any $x>0$ and $\alpha= c_0(n,k) $, see Proposition  \ref{thmA1}.
For general case, we prove that for each $n$ and any fixed $k$ with $1\leq k\leq n$, we can find a constant $ c_0(n,k) >\frac{1}{k}$ such that
$Q$ is non-positive for any $x>0$ and $\alpha= c_0(n,k) $. Although we cannot write down $ c_0(n,k) $ in term of an explicit function of $n$ and $k$, $ c_0(n,k) $  can be precisely determined by  applying Sturm's theorem,  see Proposition \ref{thmA2}. We list some of the  values of $ c_0(n,k)$ (see \eqref{listc0}). We also give some estimates of the constant $ c_0(n,k) $. We prove that $ c_0(n,1)\ge 1+\frac{7}{n}$ and $ c_0(n,k) \geq \frac{1}{k}+\frac{k}{(k-1)n}$ for $k\ge 2$ and $n\geq k^2$, see Proposition  \ref{thmA3} and Proposition \ref{thmA4}. The estimate for the case with $k\geq 2$ and $n\geq k^2$ is optimal in the sense that when $n=k^2$, the lower bound $\frac{1}{k}+\frac{k}{(k-1)n}$ equals $\frac{1}{k-1}$, which is an upper bound for $ c_0(n,k) $ (see \eqref{estalpha}).

First, in the case $1<k\le n \le k^2$, we have the following result.
\begin{prop}\label{thmA1}
If $1<k\le n \le k^2$, $\alpha=\frac{1}{k-1}$, then all the coefficients of $Q=Q(x)$ defined by \eqref{Q} are non-positive, which implies that $Q$ is non-positive for any $x>0$.
\end{prop}
\begin{proof}
We denote the coefficients of $Q=Q(x)$ by $c_i$, i.e., we write
\[
Q(x)=c_6 x^6+c_5 x^5+ c_4 x^4+c_3 x^3+ c_2 x^2 +c_1 x^1 +c_0.
\]
We will prove that if $1<k\le n \le k^2$, $\alpha=\frac{1}{k-1}$, then $c_i\leq 0$, $i=0,1,\cdots,6$.

If $1<k\le n \le k^2$, $\alpha=\frac{1}{k-1}$, it is obvious that $c_6=0$ and $c_0$ is non-positive. We will estimate $c_1$ to $c_5$ one by one.
For $c_1$, if $1<k\le n \le k^2$, $\alpha=\frac{1}{k-1}$, then we have that
\begin{equation*}
\begin{aligned}
c_1&= \frac{(n-k)^2(n-6k^2+8k)}{k-1}\leq \frac{(n-k)^2(k^2-6k^2+8k)}{k-1}\\
&=\frac{(n-k)^2(-5k^2+8k)}{k-1}\leq0.
\end{aligned}
\end{equation*}

For $c_2$, note that from the expression of $c_2$ we know that $c_2=0$ if $n=k$, hence we only need to consider that case that $k+1\leq n\leq k^2$.
If $2<k+1\le n \le k^2$, $\alpha=\frac{1}{k-1}$, then we have
\begin{equation*}
\begin{aligned}
c_2&= (k-n)\big(\alpha^2 k^2(k-2n+3)+\alpha k(6k^2-k(3n+22)+12n+3)+3k(n+1)-4n\big)\\
&=\frac{k-n}{(k-1)^2}(2k^2(3k^2-12k+11)+n(k+1)(3k-4))\\
&\leq \frac{k-n}{(k-1)^2}(2k^2(3k^2-12k+11)+(k+1)(k+1)(3k-4))\\
&= \frac{k-n}{k-1}(6k^3-15k^2+9k+4)<0.
\end{aligned}
\end{equation*}

For $c_3$, we can regard $c_3$ as a quadratic polynomial of $n$.
If $1<k\le n \le k^2$, $\alpha=\frac{1}{k-1}$, then the coefficient of $n^2$ is $-2(\alpha k-1)(2\alpha k-3)=\frac{2(k-3)}{(k-1)^2}$, which is non-negative when $k\geq 3$, which means $c_3$ is a convex function of $n$ when $k\geq 3$.
We first consider the case $k\geq 3$. In this case, it suffices to prove that both $c_3|_{n=k}$ and $c_3|_{n=k^2}$ are non-positive.
We have that
\begin{equation*}
\begin{aligned}
c_3|_{n=k}=-\frac{k^2(k-2)(2k-3)}{k-1}<0,~~c_3|_{n=k^2}=-\frac{k^3(9k-16)}{k-1} < 0.
\end{aligned}
\end{equation*}
In the case  $k=2$, we have $\alpha=1$  and $c_3 =-2n(n-2)\leq 0$.

Similarly, for $c_4$, we can regard it as a quadratic polynomial of $n$.
Note that the coefficient of $n^2$ is $2(\alpha k-1)^2$, which means that $c_4$ is a convex function of $n$.
For any $k$ and $\alpha$, the maximum of $c_4$ is attained at either $n=k$ or $n=k^2$. We only need to prove that both $c_4|_{n=k}$ and$c_4|_{n=k^2}$ are non-positive.
If $1<k\le n \le k^2$, $\alpha=\frac{1}{k-1}$, then we have
\begin{equation*}
\begin{aligned}
c_4|_{n=k}=-\frac{5k^2(k-2)}{k-1}\leq 0,~~c_4|_{n=k^2}=-\frac{5k^3}{k-1}\leq0.
\end{aligned}
\end{equation*}

For $c_5$, if $1<k\le n \le k^2$, $\alpha=\frac{1}{k-1}$, then we have $c_5=\frac{4k(n-k^2)}{(k-1)^2}\leq 0$.

Therefore,  we have proved that if $1<k\le n \le k^2$, $\alpha=\frac{1}{k-1}$, then $c_i \le0$, for $i=0,1,\cdots,6$. Consequently, we obtain that the polynomial $Q$ defined by \eqref{Q} is non-positive for any $x>0$ if $1<k\le n \le k^2$ and $\alpha=\frac{1}{k-1}$.
\end{proof}

For general cases, we apply Sturm's theorem to prove the following proposition.
\begin{prop}\label{thmA2}
For each $n$ and any fixed $k$ with $1\leq k\leq n$, we can find a constant $ c_0(n,k) >\frac{1}{k}$ such that
$Q$ defined by \eqref{Q} is non-positive for any $x>0$ and $\alpha= c_0(n,k) $.
\end{prop}
\begin{proof}
First, we define a standard sequence of a polynomial $p(x)\in\mathbb{R}[x]$ of positive degree (cf. Chapter 5.2 of  \cite{N1985BOOK}) by applying Euclid's algorithm to $p(x)$ and $p'(x)$:
\begin{equation}\label{ssp}
\begin{aligned}
p_0(x) := & p(x), \\
p_1(x) := & p'(x), \\
p_2(x):= & -\text{rem}(p_0,p_1), \\
p_3(x):= & -\text{rem}(p_1,p_2), \\
\vdots &  \\
0= & -\text{rem}(p_{m-1},p_m),
\end{aligned}\end{equation}
where $\text{rem}(p_i,p_j)$ is the polynomial remainder of the polynomial  long division of $p_i$ by $p_j$.
We call the above sequence of polynomials \textit{the standard Sturm sequence} of $p(x)$.
We will apply the following theorem.

\textbf{Sturm's theorem:} (cf. \cite{N1985BOOK})
Let $p_0(x), \dots, p_m(x)$ be the standard Sturm sequence of  a  polynomial  $p(x)\in\mathbb{R}[x]$ with positive degree. Assume that $[a,b]$ is an interval such that $p(a)p(b)\neq0$,  and let $\sigma(\xi)$ denote the number of sign changes (ignoring zeroes) in the sequence
$$
\{p_0(\xi),p_1(\xi),\dots,p_m(\xi)\},
$$
then the number of distinct roots of $p(x)$ in $(a,b)$ is $\sigma(a)-\sigma(b)$.

For each $n$ and any fixed $k$ with $1\leq k\leq n$, it is obvious that $Q<0$ if $x=0$.
Therefore, in order to prove Proposition \ref{thmA2}, by applying  Sturm's theorem, we only need to prove that for each $n$ and any fixed $k$ with $1\leq k\leq n$, we can find a constant $ c_0(n,k) >\frac{1}{k}$ such that the number $\sigma(0)- \sigma(\infty)$ of  the polynomial $Q(x, k, n,  c_0(n,k) $ equals $0$.
We will describe how to use the computer program Mathematica to help us to find the constant $ c_0(n,k) $, by using a method of bisection and applying Sturm's theorem. We can use Mathematical algorithm to run the following procedure:
For each $n$ and fixed $k$ with $1\leq k\leq n$, we fix an arbitrary precision $\delta$ and set the initial data as follows.
 	\[
	\alpha^{\text{initial}}_{min} =\frac{1}{k}, \alpha^{\text{initial}}_{max}=6 ~~\text{for}~~ k=1, \alpha^{\text{initial}}_{max}=\frac{1}{k-1}+\delta ~~\text{for}~~ k\geq 2.
	\]
	Whenever $\alpha_{max}-\alpha_{min} \ge \delta$, we do the following loop:\\
	(1) Set $\alpha_{test}=\frac{1}{2}(\alpha_{max}+\alpha_{min})$.\\
	(2) Use Euclid's algorithm to compute the Sturm sequence for the polynomial $Q(x, k, n, \alpha_{test})$.\\
	(3) Compute $\sigma(0)- \sigma(\infty)$ of  the polynomial $Q(x, k, n, \alpha_{test})$,  if $\sigma(0)- \sigma(\infty)=0$, then we set $\alpha_{min}=\alpha_{test}$, otherwise, we set $\alpha_{max}=\alpha_{test}$.\\	
Once the loop ends, we obtain two constants $\alpha_{min}(n,k)$ and $\alpha_{max}(n,k)$ which satisfy that $\alpha_{max}(n,k)-\alpha_{min}(n,k)$ is less than the given precision $\delta$, and the number $\sigma(0)- \sigma(\infty)$ of  the polynomial $Q(x, k, n, \alpha_{min}(n,k)$ equals $0$.

Therefore, $ c_0(n,k) =\alpha_{min}(n,k)$ is the constant we seek for.
\end{proof}

We list some values of  $ c_0(n,k) $ for some specific $k$, $n$, with precision $\delta=0.01$:
\begin{equation}\label{listc0}
\begin{aligned}
& c_0(3,1)=3.64...,~ c_0(4,1)=2.93...,~ c_0(5,1)=2.56...,~ c_0(6,1)=2.33...,\\
& c_0(7,1)=2.17...,~ c_0(8,1)=2.05...,~ c_0(9,1)=1.96...,~ c_0(10,1)=1.89...,\\
& c_0(11,1)=1.83...,~ c_0(12,1)=1.78...,~ c_0(3,2)= c_0(4,2)=1.
\end{aligned}
\end{equation}
In the following, we give more details of how to apply Sturm's theorem by proving the following estimate for $ c_0(n,1)$.

\begin{prop}\label{thmA3}
Let $p_0(x), \dots, p_m(x)$ be the Sturm sequence  given by \eqref{ssp} for the polynomial $p(x)=Q(x,1,n, \alpha)$ defined by \eqref{Q} with regard to $x$, and let $\sigma(\xi)$ denote the number of sign changes (ignoring zeroes) in the sequence
\[
\{p_0(\xi),p_1(\xi),\dots,p_m(\xi)\}.
\]
If $n > 12$ and $\alpha=1+\frac{7}{n}$, then $\sigma(0)- \sigma(\infty) = 0$, which means that $Q(x,1,n, \alpha)$ has no root on $(0,\infty)$ if $\alpha=1+\frac{7}{n}$. Consequently, we obtain that $Q(x,1,n, 1+\frac{7}{n}) < 0$ for all $x\in (0,\infty)$ and $n>12$, since $Q<0$ when $x=0$.
If $3\le n\le 12$, one can easily check that $ c_0(n,1)\ge 1+\frac{7}{n}$.
Therefore, we have  $ c_0(n,1)\geq 1+\frac{7}{n}$.
\end{prop}
\begin{proof}
When $n > 12$ and $\alpha=1+\frac{7}{n}$, we can divide the standard Sturm sequence of $Q(x,1,n, 1+\frac{7}{n})$ by some suitable positive functions to obtain a simpler Sturm sequence, as  Sturm's theorem only concerns the sign of each term in the Sturm sequence.
We still denote the simpler Sturm sequence by
\[
\{p_0(x),p_1(x),\dots,p_m(x)\}.
\]
Although the expression of $p_i(x)$ might be very complicated, we only need to know the signs of $p_i(0)$ and $p_i(\infty)$, that is, the signs of the zero-order terms and the
coefficients of the highest order terms of $p_i(x)$. We divide the  zero-order terms and the highest order terms of $p_i(x)$ by some suitable positive functions, and denote the remaining terms by  $Z_i$ and $I_i$, respectively. Therefore, $Z_i$ has the same sign as the zero-order term of $p_i(x)$ and $I_i$ has the same sign as the coefficient of the highest order term of $p_i(x)$.  We have
\begin{equation*}\begin{aligned}
Z_0=&7 - 19 n + 15 n^2 - n^3 - 2 n^4,~Z_1=1,\\
Z_2=&2744 n - 12348 n^2 + 18172 n^3 - 8453 n^4 - 1794 n^5 + 1535 n^6 +
 144 n^7,\\
 Z_3=&16672544 n - 60658864 n^2 + 78969576 n^3 - 38201184 n^4 -
 2317896 n^5 \\
 &+ 6372732 n^6 - 576295 n^7 - 270278 n^8 + 6081 n^9 +
 3584 n^{10},\\
  Z_4=&2529924096 - 11497601472 n + 20565314112 n^2 - 18321051392 n^3 +
 8896937056 n^4\\
 &- 2856559664 n^5 + 619264184 n^6 + 142496512 n^7 -
 32213404 n^8- 29801085 n^9  \\
 &- 21819236 n^{10} + 4720867 n^{11} +
 725886 n^{12} - 296460 n^{13} - 40000 n^{14},\\
   Z_5=&-165288374272 + 822014504192 n - 1439948464640 n^2 +
 635221775104 n^3 \\
 &+ 1099756498624 n^4 - 1442680610560 n^5 +
 317014594400 n^6 + 331846621568 n^7\\
 &- 142328426016 n^8 -
 34283941676 n^9 + 15836869820 n^{10} + 3675343420 n^{11}\\
 & - 667512847 n^{12} - 214699395 n^{13} + 31704806 n^{14} + 13123168 n^{15} +
 994304 n^{16},
 \end{aligned}\end{equation*}
\begin{equation*}\begin{aligned}
Z_6=&330576748544 - 422075670016 n - 593003666752 n^2 + 717369012864 n^3 \\
 &+ 326600077888 n^4 - 375952652096 n^5 - 61529189456 n^6 \\
 &+ 54892083792 n^7 + 9999889760 n^8 - 2733091200 n^9 - 549429077 n^{10}\\
 &+ 149191472 n^{11} + 43911424 n^{12} + 2985984 n^{13}= I_6,\\
  I_0=&-1, ~I_1=-1,~
  I_2=-192080 + 263424 n - 60368 n^2 - 13272 n^3 - 53 n^4 + 144 n^5, \\
 I_3=&90354432 - 180708864 n + 96693072 n^2 - 9686320 n^3 + 6803552 n^4 -
 3159968 n^5 \\
 &+ 71104 n^6 + 240196 n^7 + 8186 n^8 - 2400 n^9,\\
 I_4=&2529924096 - 16557449664 n + 37348649856 n^2 - 36713556544 n^3 +
 13721909824 n^4\\
 &- 526931104 n^5  + 1003345392 n^6 - 687439728 n^7 -
 136384936 n^8 + 21404941 n^9 \\
 &+ 7869536 n^{10} - 1740356 n^{11} -
 780868 n^{12} - 68800 n^{13},\\
  I_5=&165288374272 - 822014504192 n + 1463561089536 n^2 -
 741689414144 n^3 \\
 &- 920500835072 n^4 + 1357950741952 n^5 -
 451152740416 n^6 - 132175466304 n^7 \\
 &+ 73039703968 n^8 +
 20211061820 n^9 - 9966788912 n^{10} - 1997678860 n^{11} \\
 & +392467304 n^{12} + 110406478 n^{13} - 19057514 n^{14} - 7572032 n^{15} -
 594432 n^{16}.
 \end{aligned}\end{equation*}
Note that for sufficiently large $n$, the signs of $\{Z_i\}$ are
$
-,+,+,+,-,+,+,
$
and the signs of $\{I_i\}$ are
$
-,-,+,-,-,-,+.
$
This implies that $\sigma(0)= \sigma(\infty)=3$. So it remains to show that for any $i\in\{0,1,\cdots,6\}$, all real roots of $Z_i$ and $I_i$ are not greater than $12$.
Since  $Z_i$ and $I_i$ are unary polynomials, this can be proved directly  by applying  Sturm's theorem.
For example, we show how to apply Sturm's theorem to $I_2$:
\[
I_2=-192080 + 263424 n - 60368 n^2 - 13272 n^3 - 53 n^4 + 144 n^5.
\]
First, we can obtain a Sturm sequence  by using  Euclid's algorithm \eqref{ssp} and removing some positive coefficients to get a simpler Sturm sequence $\{q_i\}$. Then we have
\begin{equation}\begin{aligned}
q_0(n)=&-192080 + 263424 n - 60368 n^2 - 13272 n^3 - 53 n^4 + 144 n^5, \\
q_1(n)=&263424 - 120736 n - 39816 n^2 - 212 n^3 + 720 n^4, \\
q_2(n)=&169381632 - 188065528 n + 33126282 n^2 + 4780729 n^3, \\
q_3(n)=&-14501462505796+11364288885852 n-795070863791 n^2, \\
q_4(n)=&11296812839226538 - 7895204048274613 n,\\
q_5(n)=&-1.
\end{aligned}\end{equation}
Therefore the signs of $\{q_i(12)\}$ are
$
+,+,+,+,-,-,
$
and the signs of $\{q_i(\infty)\}$ are
$
+,+,+,-,$\\$-,-.
$
This implies that $\sigma(12)= \sigma(\infty)=1$. By applying Sturm's theorem, $I_2$ has no real root greater than $12$.
In a similar way, we obtain that all real roots of $Z_i$ and $I_i$ are not greater than $12$.
Hence, we obtain that the difference of the numbers of sign changes $\sigma(0)- \sigma(\infty)$ equals $0$ for  $Q(x,1,n, 1+\frac{7}{n})$ regarded as a polynomial of $x$, hence we obtain that   $Q(x,1,n, 1+\frac{7}{n}) < 0$ for all $x\in (0,\infty)$, since $Q<0$ when $x=0$.
\end{proof}

When $k\geq 2$, we have the following estimate for $ c_0(n,k) $.
\begin{prop}\label{thmA4}
If $k\geq 2$, $n\geq k^2$, then we have $ c_0(n,k)  \ge \frac{1}{k}+\frac{k}{(k-1)n}$.
\end{prop}
\begin{proof}
First, from Proposition \ref{thmA1}, when $n=k^2$, we have $ c_0(n,k) =\frac{1}{k-1}= \frac{1}{k}+\frac{k}{(k-1)n}$.
Second, we can use the method of bisection as described in the proof of Proposition \ref{thmA2} to estimate $ c_0(n,2)$ and $c_0(n,3)$,  and obtain that $ c_0(n,2)\ge \frac{1}{2}+\frac{2}{n}$, $\forall ~n\in [4,44]$;  $ c_0(n,3)\ge \frac{1}{3}+\frac{3}{2n}$, $\forall ~n\in [9,15]$.
In the remaining cases, i.e., $k=2, n>44$, or $k=3,~n>15$, or $k\geq 4, ~n\ge k^2+1$,
in order to prove Proposition \ref{thmA4}, we only need to show that
\[
Q(x,k,n,\frac{1}{k}+\frac{k}{(k-1)n}) \le 0, \quad\forall ~ x>0.
\]
For convenience, we reduce $Q(x,k,n,\frac{1}{k}+\frac{k}{(k-1)n})$ to a simpler polynomial:
\[
Q(x,k,n,\frac{1}{k}+\frac{k}{(k-1)n})=\frac{1}{n^2(k-1)^2} \sum_{i=0}^{6}a_i x^i,
\]
where
\begin{equation*}
\begin{aligned}
a_0=&-n(n-k)^3(k-1)(2n(k-1)+k^2),\\
a_1=&-n(n-k)^2(k-1)(n(5k^2-12k+6)+k^2(4k-6)),\\
a_2=&-6n^4(k-1)^2+n^3(k-1)(-3k^3+22k^2-30k+6)\\
&+n^2 k (-3k^4+9k^3+11k^2-21k+6)+n k^3(6k^3-29k^2+26k-9)+k^5(k+3),
\end{aligned}
\end{equation*}
and
\begin{equation*}
\begin{aligned}
a_3=&-2n^3(k-1)(5k^2-12k+6)+n^2(k^5-5k^4-17k^3+37k^2-22k+2)\\
&+n k^2(-4k^4+37k^3-45k^2+32k-4)+2k^4(-2k^2-5k+1),\\
a_4=&-6n^3(k-1)^2+n^2(15k^3-37k^2+30k-6)\\
&+n k^2(k^4-22k^3+37k^2-42k+12)+6k^4(k^2+2k-1),\\
a_5=&(n-k^2)(n(-5k^3+17k^2-18k+6)+4k^4+6k^3-6k^2),\\
a_6=&-(k-1) (n-k^2) ((k+2) k^2+2 (k-1) n).
\end{aligned}
\end{equation*}

When $k\geq 2$, we have $5k^2-12k+6>0$, and it is obvious that $a_0<0$, $a_1< 0$ and $a_6< 0$ for $n\geq k^2+1$.
Fix any $j\in \{2,3,4\}$, we regard $a_j$ as a polynomial of $n$, and  we have that $a_j \to -\infty$ as $n \to \infty$. If we set $n=k^2+1$, then $a_j$ is a unary polynomial of $k$, and we can  prove that $a_j|_{n=k^2+1}< 0$ for $k\geq 2$ directly by applying  Sturm's theorem. Finally, we can apply Sturm's Theorem to show that $a_j$ (with regard to $n$) has no roots in $[k^2+1,\infty)$, hence $a_j< 0$, if $j\in \{2,3,4\}$, $k\geq 2$ and $n\geq k^2+1$.
For $a_5$, we need to discuss three cases.

(i) If $k=2$, then $a_5=-2(n-4)(n-44)$. Hence, if $n>44$, then $a_5< 0$.

(ii) If $k=3$, then $a_5=-6(n-9)(5n-72)$. Hence, if $n>15$, then $a_5< 0$.

(iii) If $k\ge4$, since $n\ge k^2+1$, we have
\begin{equation*}
\begin{aligned}
a_5 &\le(n-k^2)((k^2+1)(-5k^3+17k^2-18k+6)+4k^4+6k^3-6k^2)\\
&=-(n-k^2)(k(k-4)+2)(5k^3-k^2+3k-3)<0.
\end{aligned}
\end{equation*}

Therefore, we have proved that all the coefficients $a_0,\cdots,a_6$ are non-positive when $k=2, n>44$, or $k=3,~n>15$, or $k\geq 4, ~n\ge k^2+1$. Consequently, we obtain that
$$Q(x,k,n,\frac{1}{k}+\frac{k}{(k-1)n}) \le 0, \quad\forall ~ x>0,$$
when $k=2, n>44$, or $k=3,~n>15$, or $k\geq 4, ~n\ge k^2+1$.
This completes the proof of Proposition \ref{thmA4}.
\end{proof}

\end{document}